\documentclass[11pt]{article}
\usepackage{cite}
\usepackage{mathrsfs} 
\usepackage{ifthen}
\usepackage{fancyhdr}
\usepackage[us,12hr]{datetime} 
\usepackage{exscale}
\usepackage{tabularx}
\usepackage{latexsym}
\usepackage{amsmath,amssymb,amstext,amsthm} 
\usepackage{graphicx,color,epsfig}
\usepackage[table]{xcolor}
\usepackage{graphicx}
\usepackage{fullpage}        
\usepackage{float}
\usepackage{verbatim}
\usepackage{enumerate}
\usepackage[shortlabels]{enumitem}
\usepackage[bookmarks,pagebackref,
    pdfpagelabels=true, 
    ]{hyperref}
    \numberwithin{equation}{section}
    \numberwithin{figure}{section}
\usepackage{makeidx}
\usepackage{tikz}
\usepackage[final]{pdfpages}
\usepackage{algorithm,algorithmic} 
\usepackage{lineno}
\usepackage{caption}
\numberwithin{table}{section}
\numberwithin{algorithm}{section}
\makeindex

\def\R{\mathbb{R}}
\def\Sc{\mathbb{S}}
\def\Sn{\Sc^n}

\def\Snp{\Sc_+^n}
\def\Srp{\Sc_+^r}

\def\Snpp{\Sc_{++}^n}

\def\Rn{\mathbb{R}^n}

\def\eqref#1{{\normalfont(\ref{#1})}}

\def\eqref#1{{\normalfont(\ref{#1})}}

\newtheorem{theorem}{Theorem}[section]

\newtheorem{assump}[theorem]{Assumption}
\newtheorem{prop}[theorem]{Proposition}

\newtheorem{cor}[theorem]{Corollary}
\newtheorem{corollary}[theorem]{Corollary}

\newtheorem{remark}[theorem]{Remark}

\newtheorem{lemma}[theorem]{Lemma}

\newtheorem{fact}[theorem]{Fact}

\newcommand{\textdef}[1]{\textit{#1}\index{#1}}

\newcommand{\Xa}{{X(\alpha)}}
\newcommand{\ya}{{y(\alpha)}}
\newcommand{\Za}{{Z(\alpha)}}

\newcommand{\LL}{{\mathcal L} }

\newcommand{\OO}{{\mathcal O} }

\newcommand{\Rm}{{\R^m\,}}

\newcommand{\NN}{{\mathbb N}}

\newcommand{\A}{{\mathcal A}}

\newcommand{\II}{{\mathcal I}}

\newcommand{\F}{{\mathcal F\,}}

\newcommand{\bbm}{\begin{bmatrix}}
\newcommand{\ebm}{\end{bmatrix}}
\newcommand{\bem}{\begin{pmatrix}}
\newcommand{\eem}{\end{pmatrix}}
\newcommand{\beq}{\begin{equation}}
\newcommand{\beqs}{\begin{equation*}}
\newcommand{\bet}{\begin{table}}
\newcommand{\eeq}{\end{equation}}
\newcommand{\eeqs}{\end{equation*}}
\newcommand{\beqr}{\begin{eqnarray}}

\DeclareMathOperator{\face}{face}
\DeclareMathOperator{\sd}{sd}

\DeclareMathOperator{\range}{range}

\DeclareMathOperator{\dist}{dist}

\DeclareMathOperator{\trace}{{trace}}

\DeclareMathOperator{\relint}{{relint}}

\DeclareMathOperator{\rank}{{rank}}
\DeclareMathOperator{\spanl}{{span}}

\newcommand{\nc}{\newcommand}
\nc{\arrow}{{\rm arrow\,}}
\nc{\Arrow}{{\rm Arrow\,}}
\nc{\BoDiag}{{\rm B^0Diag\,}}
\nc{\bodiag}{{\rm b^0diag\,}}

\nc{\Mm}{{\mathcal M}^{m} }
\nc{\Mmn}{{\mathcal M}^{mn} }
\nc{\Mnr}{{\mathcal M}_{nr} }
\nc{\Mnmr}{{\mathcal M}_{(n-1)r} }
\nc{\kwqqp}{Q{$^2$}P\,}
\nc{\kwqqps}{Q{$^2$}Ps}

\nc{\notinaho}{(X,S)\in \overline{AHO}(\A)}
\nc{\inaho}{(X,S)\in AHO(\A)}

\newcommand{\bea}{\begin{eqnarray}}%
\newcommand{\eea}{\end{eqnarray}}%
\newcommand{\beas}{\begin{eqnarray*}}%
\newcommand{\eeas}{\end{eqnarray*}}%
\renewcommand{\F}{\mathcal{F}}%
{}

\newcommand{\Hnp}[1][]{\,\mathbb{H}_+^{\ifthenelse{\equal{#1}{}}{n}{#1}}}
\newcommand{\Hn}[1][]{\,\mathbb{H}^{\ifthenelse{\equal{#1}{}}{n}{#1}}}
\newcommand{\Dn}[1][]{\,\mathbb{D}^{\ifthenelse{\equal{#1}{}}{n}{#1}}}

\newcommand{\ba}{b(\alpha) }

\newcommand{\Fa}{\F(\alpha)}

\begin{document}

\title{
Error Bounds and Singularity Degree in Semidefinite Programming
}
             \author{
\href{https://uwaterloo.ca/combinatorics-and-optimization/about/people/ssremac}{Stefan
Sremac}\thanks{Department of Combinatorics and Optimization
        Faculty of Mathematics, University of Waterloo, Waterloo, Ontario, Canada N2L 3G1; Research supported by The Natural Sciences and Engineering Research Council of Canada and by AFOSR.
}
\and
\href{http://www.math.drexel.edu/~hugo/}
{Hugo J. Woerdeman}\thanks{Department of Mathematics, Drexel University, 
3141 Chestnut Street, Philadelphia, PA 19104, USA.  Research supported by 
Simons Foundation grant 355645.
}
\and
\href{http://www.math.uwaterloo.ca/~hwolkowi/}
{Henry Wolkowicz}%
        \thanks{Department of Combinatorics and Optimization
        Faculty of Mathematics, University of Waterloo, Waterloo, Ontario, Canada N2L 3G1; Research supported by The Natural Sciences and Engineering Research Council of Canada and by AFOSR; 
\url{www.math.uwaterloo.ca/\~hwolkowi}.
}
}
          \maketitle



\vspace{-.4in}

\begin{abstract}
In semidefinite programming a proposed optimal solution may be quite
poor in spite of having sufficiently small residual in the
optimality conditions.  This issue may be framed in terms of the discrepancy between
\emph{forward error} (the unmeasurable `true error') and \emph{backward
error} (the measurable violation of optimality conditions). 
In~\cite{S98lmi}, Sturm provided an upper bound on forward error in terms of backward error and \emph{singularity degree}.  In this work we provide a method to bound the maximum rank over all solutions and use this result to obtain a lower bound on forward error for a class of convergent sequences.  This lower bound complements the upper bound of Sturm.  The results of Sturm imply that semidefinite programs with slow convergence necessarily have large singularity degree.  Here we show that large singularity degree is, in some sense, also a sufficient condition for slow convergence for a family of external-type `central' paths.  Our results are supported by numerical observations.

\end{abstract}




\section{Introduction}
\label{sec:intro}

It is well known that for certain pathological instances of semidefinite
programming, state-of-the-art algorithms, while theoretically guaranteed
to converge to a solution, do so very slowly or can fail to converge entirely.  
This issue is exacerbated in that it is generally undetectable.  In this
paper we propose a method to detect this type of slow convergence by lower bounding \emph{forward error}, i.e.,~distance to the solution set.  This bound is obtained by analyzing a class of parametric curves  that are proven to converge to a solution of maximum rank and then upper bounding that rank.  In the second part of the paper we present a new analysis of the relation between forward error and \emph{singularity degree}, a measure introduced by Sturm in~\cite{S98lmi} and shown to be a necessary condition for slow convergence.  Our results indicate that large singularity degree is, in some sense, also a sufficient condition for slow convergence for a certain family of central paths.

\index{SDP, semidefinite program}
\index{semidefinite program, SDP}
\index{$\Sn$, Euclidean space of $n\times n$ symmetric matrices}
\index{Euclidean space of $n\times n$ symmetric matrices, $\ \Sn$}
\index{$\Snp$, set of positive semidefinite matrices}
\index{set of positive semidefinite matrices, $\Snp$}
\index{$\F$, spectrahedron}
\index{spectrahedron, $\F$}
To be more specific about the type of slow convergence we are concerned with, let ${\F \subset \Sn}$ be the solution set of a \emph{semidefinite program (SDP)}.  Throughout this paper we refer to $\F$ as a \emph{spectrahedron}.  Here $\Sn$ denotes the ambient space of $n\times n$ symmetric matrices.  It is always possible to express $\F$ as the intersection of an affine subspace, $\LL$, and the set of positive semidefinite matrices, $\Snp$.  Given a matrix $X \in \Sn$, the \emph{forward error} is defined as,
\index{$\epsilon^f$, forward error}
\index{forward error, $\epsilon^f$}
\begin{equation}
\label{eq:forward}
\epsilon^f(X,\F) := \dist(X, \F).
\end{equation}
We cannot expect to measure forward error accurately without substantial knowledge of $\F$.  For this reason forward error is generally unknown.  What is readily available to users is \emph{backward error}, 
\index{backward error, $\epsilon^b$}
\index{$\epsilon^b$, backward error}
\begin{equation}
\label{eq:backward}
\begin{split}
\epsilon^b(X,\F) &:= \dist(X, \LL) + \dist(X,\Snp).
\end{split}
\end{equation}
In backward error we recognize that $\F$ is the intersection of two sets with easily computable forward errors.  Backward error serves as a proxy for the unknown forward error.  The type of slow convergence we are concerned with is when \emph{backward error is sufficiently small but forward error is much larger}. The problem with this scenario is not just the poor quality of the proposed solution.  More than this, it is the lack of awareness of a poor solution.

To demonstrate the discrepancy between forward error and backward error, we consider an SDP, with $n=5$, from the family introduced in~\cite{MR2724357}.  The output of \texttt{cvx} using the solver \texttt{SDPT3}
is,
\index{$\lambda(X)$, vector of eigenvalues of $X$}
\index{vector of eigenvalues of $X$, $\lambda(X)$}
\[
X \approx \begin{bmatrix}
0.94         &   0  & 0.028  & 0.001 & 2.3\times 10^{-6}\\
            0   & 0.057         &   0    &        0    &        0 \\
   0.028     &       0 &  0.028 &  4.1\times 10^{-5} &  6.5\times 10^{-8}\\
   0.001      &      0 &  4.1\times 10^{-5} &  4.5\times 10^{-6} &  3.1\times 10^{-9}\\
   2.3\times 10^{-6}      &      0  & 6.5\times 10^{-8} &  3.1\times 10^{-9}    &        0
\end{bmatrix}, \ \lambda(X) \approx \begin{pmatrix}
0.94 \\
   0.057\\
   1.9\times 10^{-3}\\
   2.4\times 10^{-6}\\
  -5.4\times 10^{-12}
\end{pmatrix},
\]
where $\lambda(X)$ is the vector of eigenvalues of $X$.  Similar results were obtained with the solvers \texttt{SeDuMi} and \texttt{MOSEK}.  The backward error for $X$ is quite small at~${5.46\times 10^{-12}}$ and \texttt{cvx} output states that the problem is ``solved''.  All indicators point to a `good' solution. However, the solution set of the SDP is a singleton consisting of the matrix with $1$ in the upper left entry and zeros everywhere else.  Given this information, $X$ does \emph{not} look like a very good solution.  Indeed, forward error is ${9.15\times 10^{-2}}$, unacceptably large.  Moreover, the eigenvalues of $X$ hardly indicate that the solution is a rank one matrix.  In the numerical case studies of Section~\ref{sec:numerics}, we show that our lower bound on forward error is significantly greater than backward error.  Therefore, it serves as an alarm that the proposed solution is not as accurate as it appears to be.
  
In~\cite{S98lmi}, Sturm defined singularity degree as the fewest number of iterations required in the \emph{facial reduction algorithm}, a regularization scheme for conic optimization introduced in~\cite{bw1,bw2,bw3}.  The singularity degree of a spectrahedron $\F$, denoted $\sd(\F)$, is an integer between $0$ and $n-1$.  Sturm showed that forward error is bounded in terms of backward error and singularity degree,
\index{$\sd(\F)$, singularity degree of $\F$}
\index{singularity degree of $\F$, $\sd(\F)$}
\begin{equation}
\label{eq:sturmbound}
\epsilon^f(X,\F) = \OO\left(\epsilon^b(X,\F)^{2^{-\sd(\F))}}\right).
\end{equation}
In particular, this bound implies that large singularity degree is a necessary condition for large forward error.  Equivalently, small singularity degree implies small forward error.  It is exactly this relation that has motivated our study of singularity degree.

The challenge with singularity degree is that, like forward error, it is unknown in most cases.  In~\cite{ScTuWonumeric:07} it is shown that the facial reduction algorithm is stable when singularity degree is $0$ or $1$, but the authors were not able to show that stability holds for larger singularity degree.  Moreover, the empirical evidence we have obtained indicates a lack of stability of the algorithm when singularity degree is greater than $1$.  For this reason, we view singularity degree as intractable for general instances of SDP.  Here we obtain a lower bound on singularity degree as a consequence of the upper bound on maximum rank.

Our final contribution in this work is showing that singularity degree is also a sufficient measure, in some sense, for large forward error.  We prove that for a class of central paths the eigenvalues that vanish, do so at a `fast rate' if, and only if, singularity degree is at most $1$.  We also prove that among the elements of the dual path that converge to $0$, there are at least $\sd(\F)$ different rates of convergence.

The paper is organized as follows.  In Section~\ref{sec:notback} we introduce our notation and basic concepts pertaining to facial reduction and singularity degree.  The bounds on maximum rank, forward error, and singularity degree are presented in Section~\ref{sec:boundrankferror} and in Section~\ref{sec:boundsd} we present results that support the notion that singularity degree is a measure of hardness.  The paper is concluded with numerical observations in Section~\ref{sec:numerics}.
\section{Notation and Background}
\label{sec:notback}
Throughout this paper the ambient space is the Euclidean space of
$n\times n$ real symmetric matrices, denoted $\Sn$, with the standard trace inner product,
$\langle X,Y \rangle := \trace(XY)$,
and the induced Frobenius norm,
$\lVert X \rVert_F := \sqrt{\langle X, X\rangle }$.
\index{$\lVert \cdot \rVert_F$, Frobenius norm}
\index{Frobenius norm, $\lVert \cdot \rVert_F$}
\index{$\langle \cdot, \cdot \rangle$, trace inner product on $\Sn$}

The eigenvalues of any $X \in \Sn$ are real and ordered so as to satisfy,
$\lambda_1(X)  \ge \cdots \ge \lambda_n(X)$,
and $\lambda(X) \in \Rn$ is the vector consisting of all the eigenvalues.  
In terms of this notation we have $\lVert X\rVert_F = \lVert \lambda(X) \rVert_2$, where $\lVert \cdot \rVert_2$ is the Euclidian norm when the argument is a vector in $\Rn$.  When the argument to $\lVert \cdot \rVert_2$ is a symmetric matrix then we mean the operator 2-norm, defined as
$\lVert X \rVert_2 := \max_i \lvert \lambda_i(X) \rvert$.  In some of our discussion we use the notation~${\lambda_{\max}(X) = \lambda_1(X)}$ and~${\lambda_{\min}(X) = \lambda_n(X)}$ if we are not concerned with the dimensions of the matrix, or wish to stress the minimality and maximality of the values.

The set of positive semidefinite
matrices, $\Snp$, is a closed convex cone in $\Sn$, with
interior consisting of the positive definite matrices,
$\Snpp$.  The cone $\Snp$ induces the \textdef{L\"owner partial order} on $\Sn$.
That is, for $X,Y \in \Sn$ we write $X\succeq Y$ when $X-Y \in \Snp$ and similarly $X\succ Y$ when~${X-Y \in \Snpp}$.
\index{$\lambda_i(X)$, the $i$th largest eigenvalue of $X$}

\subsection{Facial Reduction for SDPs}
\label{sec:FR}
To introduce facial reduction, we begin with a brief discussion of the faces of $\Snp$.  For further reading and proofs of some of our claims, we suggest~\cite{SaVaWo:97,MR2724357,con:70}.  A \textdef{face} of $\Snp$, say $f$, is a convex subset of $\Snp$ such that,
\[
X,Y \in \Snp, \ X+Y \in f \ \implies \ X,Y \in f.
\]
For a face $f$ there exists ${r \in \{0,\dotso,n\}}$, ${W \in \Snp}$, and ${V \in \R^{n\times r}}$ such that,
\begin{equation}
\label{eq:facechar}
WV=0, \ W+VV^T \succ 0, \ f = V\Srp V^T = \Snp \cap W^{\perp}.
\end{equation}
When the matrix $W$ in \eqref{eq:facechar} is not the $0$ matrix, it is referred to as an \textdef{exposing vector} for $f$.

\index{$\face(C)$, minimial face of $\Snp$ containing $C$}
\index{minimial face of $\Snp$ containing $C$, $\face(C)$}
An important notion regarding SDPs is that of \emph{minimal face}.  The minimal face of $\Snp$ containing a convex set $C$, denoted $\face(C)$, is the intersection of all faces of $\Snp$ that contain $C$.  If the minimal face for an SDP is known, then the SDP may be transformed into an equivalent SDP for which the \emph{Slater condition} -- strict feasibility with respect to the positive semidefinite constraint -- holds.  See the survey~\cite{DrusWolk:16} or~\cite{MR3108446,permfribergandersen,perm,MR3063940} for further reading on regularization of SDPs via facial reduction. 

For SDPs with special structure, the minimal face may be obtained through theoretical analysis.  Alternatively, the minimal face may be obtained using the \textdef{facial reduction algorithm}, introduced in~\cite{bw1,bw2,bw3}.  This algorithm generates a sequence of faces $f^1, \cdots ,f^d$ satisfying,
\[
f^1 \supsetneq \cdots \supsetneq f^d, \quad f^d = \face(\F).
\]
Equivalently, using the two characterizations of faces in \eqref{eq:facechar}, the algorithm generates a sequence of matrices ${W^1,\dotso,W^d}$, a sequence of decreasing positive integers ${r_1,\dotso,r_d}$, and a sequence of matrices ${V^k \in \R^{n\times r_k}}$ with ${k \in \{1,\dotso,d\}}$ such that,
\[
f^k = V^k \Sc^{r_k}_+ \left(V^k \right)^T = \Snp \cap \left(W^k\right)^{\perp}, \ k \in \{1,\dotso,d\}.
\]

The facial reduction algorithm depends on the algebraic representation of the solution set of an SDP.  As stated earlier, such a set is the intersection of an affine subspace and $\Snp$.  We assume that the affine subspace is defined in terms of a linear map $\A : \Sn \to \Rm$ and a vector $b \in \Rm$ so that,
\begin{equation}
\label{eq:FAb}
\F = \F(\A,b) = \{X\in \Snp : \A(X) = b\}
\end{equation}
The notation $\F(\A,b)$ stresses the dependence on the algebraic representation of the the affine subspace.  The facial reduction algorithm relies on the following theorem of the alternative.
\begin{fact}
\label{fact:thmalt}
Let $\F = \F(\A,b)$ be defined as in \eqref{eq:FAb} and nonempty.  Then exactly one of the following holds:
\begin{enumerate}[$(i)$]
\item $\F \cap \Snpp \ne \emptyset$, 
\label{itm:thmalt1}
\item there exists nonzero $W =\A^*(y)$ with $y^Tb = 0$.
\label{itm:thmalt2}
\end{enumerate}
\end{fact}
A proof of this result may be found in \cite{DrusWolk:16}, for instance.

Now let us briefly describe how this result may be used for facial reduction.  Recall that the purpose of the facial reduction algorithm is to create an equivalent representation of $\F$ so that the Slater condition holds.  Now if Fact~\ref{fact:thmalt}~\ref{itm:thmalt1} holds, then $\F$ has a Slater point and we are done.  On the other hand if Fact~\ref{fact:thmalt}~\ref{itm:thmalt2} holds, then it can be verified that the matrix $W$ is an exposing vector for a face containing $\face(\F)$.  Indeed, the definition of $W$ implies that $\langle X,W\rangle = 0$ for all~${X \in \F}$.  Letting $r$ be the nullity of $W$ and choosing $V$ so as to satisfy the properties in \eqref{eq:facechar}, we see that for all $X \in \F$ we have $X \in V\Srp V^T$.  It follows that,
\begin{align*}
\F &= \{ VRV^T \in \Snp : \A(VRV^T) = b\} \\
&= V\{R \in \Srp : \A(VRV^T) = b \}V^T \\
&= V\F \left(\A(V\cdot V^T),b \right)V^T.
\end{align*}
In the next iteration of the algorithm we apply Fact~\ref{fact:thmalt} to the spectrahedron $\F \left( \A(V\cdot V^T), b \right)$.  We continue in this way until eventually Fact~\ref{fact:thmalt}~\ref{itm:thmalt1} -- the Slater condition -- holds.  We have included Algorithm~\ref{algo:FR} as a more rigorous description of facial reduction.

\begin{algorithm}[h!]
\caption{Facial Reduction}
\label{algo:FR}
\begin{algorithmic}[1]
\STATE \textbf{INPUT:} $\A$, $b$.
\STATE \textbf{initialize:} $k=0$, $\A^k = \A$, $V^k = I$, $W^k = 0$, $r_k = n$, $q_k = 0$.
\WHILE {$\F(\A^k,b) \cap \Sc^{r_k}_{++} = \emptyset$} 
\STATE obtain nonzero $Z^{k+1} = \left( \A^k \right)^*(y^{k+1}) \succeq 0$ such that $\left(y^{k+1}\right)^Tb = 0$ and orthogonal $\begin{bmatrix} Q_1^{k+1} & Q_2^{k+1} \end{bmatrix}$ such that,
\[
Z^{k+1} = \begin{bmatrix}
Q_1^{k+1} & Q_2^{k+1}
\end{bmatrix}\begin{bmatrix}
\Lambda^{k+1} & 0 \\
0 & 0
\end{bmatrix}\begin{bmatrix}
Q_1^{k+1} & Q_2^{k+1}
\end{bmatrix}, 
\]
where $\Lambda^{k+1} \succ 0$ and $Q_1^{k+1} \in \R^{r_k \times q_{k+1}}$.
\IF {$Z^{k+1}\succ 0$}  
\STATE  $Q_2^{k+1} = 0 \in \Sc^{r_k}$, $V^{k+1} = 0 \in \Sn$, and $r_{k+1} = 0$
\ELSE
\STATE $Q_2^{k+1} \in \R^{r_k\times r_{k+1}}$ and $V^{k+1} = V^kQ_2^{k+1} \in \R^{n\times r_{k+1}}$
\ENDIF
\STATE $W^{k+1} = W^k + V^kZ^{k+1}\left(V^k\right)^T \in \Snp$
\STATE $\A^{k+1} = \A \left(V^{k+1} \cdot \left( V^{k+1} \right)^T \right)$
\STATE $k = k+1$
\ENDWHILE
\STATE \textbf{OUTPUT:} $d=k$, $V=V^k$, $W=W^k$, $r=r_d$.
\end{algorithmic}
\end{algorithm}

At each iteration, the order of the matrices is reduced by at least one, implying that the algorithm terminates in at most $n$ iterations.  In fact, the upper bound is actually $n-1$ since in the case $\F=\{0\}$ it can be shown that the algorithm terminates in exactly $1$ iteration.  

\subsection{Singularity Degree and the Bounds of Sturm}
\label{sec:sdandSturm}

The number of iterations required by the facial reduction algorithm of the previous section is dependent on the choice of exposing vector, $Z^{k+1}$, obtained at each iteration.  When the exposing vector is chosen to have maximum rank, the number of iterations generated by Algorithm~\ref{algo:FR} is defined as the \textdef{singularity degree}, and denoted $\sd(\F)$.  It can be shown that singularity degree is the least number of iterations required by the algorithm.  

Two special cases deserve mention.  The first case is that of $\sd(\F) = 0$.  This case occurs if, and only if, the solution set satisfies the Slater condition.  The second special case is when~${\F = \{0\}}$.  Here our definition of singularity degree does not coincide with that of Sturm.  By our definition we have $\sd(\F) = 1$, since exactly one iteration of the algorithm is required to obtain the Slater condition.  On the other hand, Sturm defines $\sd(\F) = 0$ for this case, based on his error bounds.  That is, the worst case error bounds when $\F=\{0\}$ are the same as when $\F$ satisfies the Slater condition. 

We now state the error bounds of Sturm.  For a proof, see Theorem~3.3 of~\cite{S98lmi}.
\begin{fact}
\label{fact:sturmbound}
Let $\F = \F(\A,b)$ be a nonempty spectrahedron and let $\{X(\alpha) : \alpha > 0\}$ be a sequence where $\lVert X(\alpha) \rVert$ is bounded. Then,
\[
\epsilon^f(\Xa, \F) = \begin{cases}
\OO \left( \epsilon^b(\Xa, \F) \right) \quad & \text{if } \F = \{0\}, \\ 
\OO \left( \epsilon^b(\Xa, \F)^{2^{-\sd(\F)}} \right) & \text{otherwise.}
\end{cases}
\]
\end{fact}
In proving Fact~\ref{fact:sturmbound}, Sturm actually obtained the following more precise statement about the way in which $\Xa$ approaches $\F$.  
\begin{fact}
\label{fact:sturmblockbounds}
Let $\F=\F(\A,b)$ be a nonempty spectrahedron with ${\sd(\F) \ge 1}$ where ${\F \ne \{0\}}$ and let~${\{\Xa : \alpha > 0\}}$ be a sequence where $\epsilon^b(\Xa, \F) = \OO(\alpha)$.  For ${i \in \{1,\dotso,\sd(\F)\}}$, let $Z^i$ be a maximum rank exposing vector obtained as in Algorithm~\ref{algo:FR} and let ${q_i := \rank(Z^i)}$.  Let $\bar \alpha > 0$ be fixed.  Then there exists an orthogonal matrix $Q$ such that,
\[
\face(Q\F Q^T) = \begin{bmatrix}
\Srp & 0 \\
0 & 0
\end{bmatrix},
\]
and,
\[
Q\Xa Q^T = \begin{bmatrix}
X_0(\alpha) & *  & \cdots &* \\
*& X_1(\alpha) & & \\
\vdots & & \ddots &*  \\
*&* & *& X_{\sd(\F)}(\alpha)
\end{bmatrix},
\]
where $X_0(\alpha) \in \Sc^r$ and for all $i \in \{1,\dotso,\sd(\F)\}$ and $\alpha \in (0,\bar \alpha)$ it holds that,
\[
X_i(\alpha) \in \Sc^{q_i} \text{ and } \lVert X_i(\alpha) \rVert = \OO \left(\alpha^{\xi(i)} \right),
\]
where $\xi(i) := 2^{-(\sd(\F)-i)}$.
\end{fact}
\index{$\xi(\cdot)$, exponent function for bound of Sturm}
\index{exponent function for bound of Sturm, $\xi(\cdot)$}
This result shows that under the correct orthogonal transformation, the diagonal blocks of~$\Xa$ that converge to $0$ may do so at different rates.

\section{Bounds on Maximum Rank, Forward Error, and Singularity Degree}
\label{sec:boundrankferror}
We consider SDPs in the form,
\begin{equation}
\label{eq:sdp}
\begin{split}
p^* := \min \ & \langle C, X \rangle \\
\text{s.t.} \ & \widehat \A(X)=\widehat b \\
& X \succeq 0,
\end{split}
\end{equation}
where $\widehat \A : \Sn \to \R^{\widehat m}$ is a linear map, $\widehat b \in \R^{\widehat m}$, and $C \in \Sn$.  The solution set is,
\begin{equation}
\label{eq:solset}
\F = \{ X \in \Snp : \widehat \A(X)=\widehat b, \ \langle C,X \rangle = p^*\}.
\end{equation}
When \eqref{eq:sdp} is a feasibility problem then $C=0$ and $p^*=0$.
In this case we let $m:= \widehat m$, $\A: = \widehat \A$, and $b :=
\widehat b$ so that ${\F = \F(\A,b) = \F(\widehat \A, \widehat b)}$.  On
the other hand, when the objective of \eqref{eq:sdp} is non-trivial,
i.e.,~$C \ne 0$, then it may not be the case that ${\F =\F(\widehat \A, \widehat b)}$.  Thus we define $m := \widehat m + 1$ and define $\A$ and $b$ as,
\[
\A(X) := \begin{pmatrix}
\widehat \A(X) \\
\langle C, X \rangle
\end{pmatrix}, \ b := \begin{pmatrix}
\widehat b \\
p^*
\end{pmatrix}.
\]
Then $\F = \F(\A,b)$ is a spectrahedron in the notation we have already developed.
\begin{assump}
\label{assump:F}
Let $\F=\F(\A,b)$ be the spectrahedron consisting of the solutions to the SDP in \eqref{eq:sdp}.  We assume that $\F$ is nonempty, ${\sd(\F) \ge 1}$, and ${\F \ne \{0\}}$.
\end{assump}
The assumption that $\F$ is nonempty is for the purpose of defining singularity degree.  The other two assumptions ensure that there is a possibility of discrepancy between forward error and backward error.

Our analysis in this section is based on \emph{path-following} algorithms for SDP.  The foundation of such algorithms is the \textdef{central path}, a smooth parametric curve, say ${\{X(\alpha) : \alpha > 0\}}$, that is known to converge to a solution of the SDP, granted that a solution exists.  Specifically we mean that ${\Xa \to \bar X}$ as ${\alpha \searrow 0}$ and ${\bar X \in \F}$.  Then a path-following algorithm produces a sequence of positive numbers $\{\alpha_k\}$ and matrices $\{ X^k\}$ such that $\alpha_k$ is successively closer to $0$ and $X^k$ is a successively better approximation of $X(\alpha_k)$.  In other words, the iterates approach the solution set along a trajectory that approximates the central path.  

To obtain a method for bounding maximum rank and forward error, we study central paths that satisfy the following assumptions.
\begin{assump}
\label{assump:para}
Let $\{X(\alpha): \alpha > 0\}$ be a central path and let $\F=\F(\A,b)$ be the spectrahedron consisting of the solutions to the SDP in \eqref{eq:sdp}.  We assume that,
\begin{enumerate}[$(i)$]
\item there exists $\bar X \in \relint(\F)$ such that $\lim_{\alpha \searrow 0} \Xa = \bar X,$
\label{itm:1para}
\item $\Xa \succ 0$ for all $\alpha >0$.
\label{itm:2para}
\end{enumerate}
\end{assump}
Many of the well-known algorithms for SDP are based on central paths that satisfy this assumption.

\subsection{A Bound on Maximum Rank}
\label{sec:boundrank}

We provide two ways to bound maximum rank.  The first method is based on tracking the ratios,
\begin{equation}
\label{eq:m1}
\frac{\lambda_i(\Xa)}{\lambda_{i+1}(\Xa)}, \ i \in \{1,\dotso,n-1\}.
\end{equation}
The ratios that blow up indicate one of two scenarios.  The first scenario is that both eigenvalues converge to $0$, but $\lambda_{i+1}(\Xa)$ does so much more quickly.  The second is that $\lambda_i(\Xa)$ converges to a positive value and $\lambda_{i+1}(\Xa)$ vanishes.  This only happens when $i$ corresponds to the rank of the limit point $\bar X$.  Thus the smallest index $i$ for which the ratio blows up corresponds exactly to the maximum rank.  We state this observation formally in the following.
\begin{prop}
\label{prop:eig}
Let ${\{\Xa : \alpha > 0\}}$ be a central path satisfying Assumption~\ref{assump:para} for a spectrahedron $\F=\F(\A,b)$ satisfying Assumption~\ref{assump:F}.  Let $i$ be the smallest index such that the ratio in \eqref{eq:m1} blows up.  Then $i$ is the maximum rank over $\F$. 
\end{prop}
In practice, a ratio of the form in \eqref{eq:m1} may blow up slowly, so that it may appear to be bounded.  In such cases an approximation of the maximum rank may be obtained by the smallest index $i$ for which it is \emph{clear that the ratio blows up}.  

In addition to differentiating between eigenvalues that vanish and those that do not, the ratios or \eqref{eq:m1} also indicate the number of different rates of convergence among the eigenvalues that do converge to $0$.  We explore this further in Section~\ref{sec:sdhard}.  

The second method to bound maximum rank is more involved, but provides us with additional information on singularity degree.  Moreover, this method appears to be more reliable in our experiments.  The approach is to analyze the $Q$-convergence rates of the eigenvalues of $\Xa$.  We begin by translating Fact~\ref{fact:sturmblockbounds} into a statement about the eigenvalues of~$\Xa$.
\begin{lemma}
\label{lem:sturmeigbounds}
Let ${\{\Xa : \alpha > 0\}}$ be a central path satisfying Assumption~\ref{assump:para} for a spectrahedron $\F=\F(\A,b)$ that satisfies Assumption~\ref{assump:F}.  Assume further that ${\epsilon^b(\Xa, \F) = \OO(\alpha)}$.  For~${i \in \{1,\dotso,\sd(\F)\}}$, let~$Z^i$ be a maximum rank exposing vector obtained as in Algorithm~\ref{algo:FR} and let ${q_i := \rank(Z^i)}$.  Let $r$ denote the maximum rank over $\F$.  Let $\II^0,\II^1,\dotso, \II^{\sd(\F)}$ form a partition of $\{1,\dotso,n\}$ such that $\II^0 = \{1,\dotso,r\}$ and,
\[
\II^1 = r + \{1,\dotso,q^1\}, \ \II^2 = r+q^1 + \{1,\dotso,q^2\}, 
\]
and so on.  Then, for $j \in \{1,\dotso,n\}$ it holds that for sufficiently small $\alpha>0$,
\[
j \in \II^i \ \implies \ \lambda_j(\Xa) = \begin{cases}
\Theta(1) \quad &\text{if } i = 0, \\
\OO \left(\alpha^{\xi(i)}\right) &\text{otherwise,}
\end{cases}
\]
where $\xi(i) := 2^{-(\sd(\F)-i)}$.
\end{lemma}
\begin{proof}
By assumption, $\Xa \to \bar X \in \relint(\F)$ and $\rank(\bar X) = r$.  Therefore for sufficiently small $\alpha > 0$, the $r$ largest eigenvalues of $\Xa$ are converging to positive numbers.  It follows that for sufficiently small $\alpha > 0$,
\[
j \in \II^0 \ \implies \ \lambda_j(\Xa) = \Theta(1),
\]
proving one part of the desired result.

Next, by Fact~\ref{fact:sturmblockbounds} there exists an orthogonal $Q$ such that,
\begin{equation}
\label{eq:1sturmeigbounds}
\face(Q\F Q^T) = \begin{bmatrix}
\Srp & 0 \\
0 & 0
\end{bmatrix} \text{ and } 
Q\Xa Q^T = \begin{bmatrix}
X_0(\alpha) & *  & \cdots &* \\
*& X_1(\alpha) & & \\
\vdots & & \ddots &*  \\
*&* & *& X_{\sd(\F)}(\alpha)
\end{bmatrix} \ \forall \alpha > 0,
\end{equation}
where $X_0(\alpha) \in \Sc^r$ and for all $i \in \{1,\dotso,\sd(\F)\}$ it holds that,
\begin{equation}
\label{eq:2sturmeigbounds}
X_i(\alpha) \in \Sc^{q_i} \text{ and } \lVert X_i(\alpha) \rVert = \OO \left(\alpha^{\xi(i)} \right).
\end{equation}
Now let $i \in \{1,\dotso,\sd(\F)\}$ and let $j \in \II^i$.  Consider the principal submatrix of $Q\Xa Q^T$,
\begin{equation}
\label{eq:3sturmeigbounds}
S(\alpha) := \begin{bmatrix}
X_i(\alpha) & & \\
  & \ddots &  \\
 & & X_{\sd(\F)}(\alpha)
\end{bmatrix}.
\end{equation}
By Assumption~\ref{assump:para}, $\Xa \succ 0$ implying that $S(\alpha) \succ 0$.  Thus by \eqref{eq:2sturmeigbounds} we have,
\[
\lVert S(\alpha) \rVert = \OO \left(\max_{\ell \in \{i,\dotso,\sd(\F)} \lVert X_{\ell}(\alpha) \rVert \right) = \OO \left(\alpha^{\xi(i)} \right).  
\]
It follows that $\lambda_1(S) = \OO \left(\alpha^{\xi(i)}\right)$.  Moreover, the interlacing eigenvalue theorem implies that,
\[
\lambda_j(\Xa) \le \lambda_1(S).
\]  
Combining these inequalities yields the desired result, $\lambda_j(\Xa) = \OO \left(\alpha^{\xi(i)}\right)$.
\end{proof}

Now that we have bounds on those eigenvalues of $\Xa$ that converge to $0$ we analyze their $Q$-convergence rates.  First a lemma.
\begin{lemma}
\label{lem:2sequence}
Let $\{a_k\}_{k\in \NN}$ and $\{b_k\}_{k\in \NN}$ be sequences of positive reals such that ${a_k \to 0}$ and ${b_k \to 0}$.  If ${a_k \le b_k}$ for all ${k \in \NN}$ then,
\begin{equation}
\label{eq:2sequence1}
\liminf_{k\to \infty} \frac{a_{k+1}}{a_k} \le \limsup_{k \to \infty} \frac{b_{k+1}}{b_k}.
\end{equation}
\end{lemma}
\begin{proof}
Let $L_a$ and $L_b$ denote the limit inferior and limit superior of \eqref{eq:2sequence1}, respectively.  For simplicity we assume that $L_a$ and $L_b$ are finite, but the arguments extend to the general case trivially.  Suppose for the sake of contradiction that there exists $\tau > 0$ such that ${L_a - \tau \ge L_b}$.  Then there exists ${\bar k \in \NN}$ such that for all ${k \ge \bar k}$,
\begin{equation}
\label{eq:2sequencemain}
\frac{a_{k+1}}{a_k} \ge L_a - \frac{\tau}{3} \text{ and } \frac{b_{k+1}}{b_k} \le L_a - \frac{\tau}{2}
\end{equation}
Rearranging the first equation in \eqref{eq:2sequencemain} gives us,
\begin{equation}
\label{eq:2sequence2}
a_{k+1} \ge a_k\left(L_{a} - \frac{\tau}{3}\right), \quad \forall k \ge \bar k.
\end{equation}
Replacing $k$ with $k-1$ we get that,
\begin{equation}
\label{eq:2sequence3}
a_{k} \ge a_{k-1}\left(L_{a} - \frac{\tau}{3}\right), \quad \forall k \ge \bar k +1.
\end{equation}
Combining \eqref{eq:2sequence2} with \eqref{eq:2sequence3} yields,
\[
a_{k+1} \ge a_{k-1}\left(L_{a} - \frac{\tau}{3}\right)^2, \quad \forall k \ge \bar k + 1.
\]
Continuing in this fashion we get,
\begin{equation}
\label{eq:2sequence4}
a_{k} \ge a_{\bar k}\left(L_{a} - \frac{\tau}{3}\right)^{k-\bar k} = \frac{a_{\bar k}}{\left(L_{a} - \frac{\tau}{3}\right)^{\bar k}}\left(L_{a} - \frac{\tau}{3}\right)^k, \quad \forall k \ge \bar k.
\end{equation}
Through an analogous approach applied to the second equation of \eqref{eq:2sequencemain} we get,
\begin{equation}
\label{eq:2sequence5}
b_k \le  \frac{b_{\bar k}}{\left(L_{a} - \frac{\tau}{2}\right)^{\bar k}}\left(L_{a} - \frac{\tau}{2}\right)^k, \quad \forall k \ge \bar k.
\end{equation}
Combining the hypothesis that $b_k$ dominates $a_k$ with \eqref{eq:2sequence4} and \eqref{eq:2sequence5} we get,
\begin{equation}
\label{eq:2sequence6}
\frac{b_{\bar k}}{\left(L_{a} - \frac{\tau}{2}\right)^{\bar k}}\left(L_{a} - \frac{\tau}{2}\right)^k \ge  \frac{a_{\bar k}}{\left(L_{a} - \frac{\tau}{3}\right)^{\bar k}}\left(L_{a} - \frac{\tau}{3}\right)^k, \quad \forall k \ge \bar k.
\end{equation}
Observe that $L_{b} \ge 0$ since $b_k \ge 0$ for every $k \in \NN$.  Therefore, ${L_a - \tau \ge 0}$ and we have ${L_{a} - \frac{\tau}{3} > L_{a} - \frac{\tau}{2} > 0}$.  It follows that for sufficiently large $k$, the inequality in \eqref{eq:2sequence6} is violated, giving us the desired contradiction.
\end{proof}

Our main result on $Q$-convergence rates of eigenvalues considers sequences of the form $\{\sigma_k\}$ for some $\sigma \in (0,1)$ where the $k$th term, $\sigma_k$, is the $k$th power of $\sigma$,~i.e.~$\sigma^k$. 
\begin{theorem}
\label{thm:eigQrate}
Let ${\{\Xa : \alpha > 0\}}$ be a central path satisfying Assumption~\ref{assump:para} for a spectrahedron $\F=\F(\A,b)$ that satisfies Assumption~\ref{assump:F}.  Assume further that ${\epsilon^b(\Xa, \F) = \OO(\alpha)}$.  Let~${\II^0,\II^1, \dotso, \II^{\sd(\F)}}$ be a partition of $\{1,\dotso,n\}$ as in Lemma~\ref{lem:sturmeigbounds}.  For ${\sigma \in (0,1)}$ the following hold.
\begin{enumerate}[$(i)$]
\item If ${j \in \II^0}$ then,
\[
\lim_{k \rightarrow \infty} \frac{\lambda_j(X(\sigma^{k+1}))}{\lambda_j(X(\sigma^k))} = 1.
\]
\label{itm:eigQrate1}
\item If ${j \in \II^i}$ with $i \in \{1,\dotso, \sd(\F)\}$ then,
\[
\liminf_{k \rightarrow \infty} \frac{\lambda_j(X(\sigma^{k+1}))}{\lambda_j(X(\sigma^k))} \le \sigma^{\xi(i)} < 1,
\]
where $\xi(i) = 2^{-(\sd(\F)-i)}$.
\label{itm:eigQrate2}
\end{enumerate}
\end{theorem}
\begin{proof}
By Assumption~\ref{assump:para} and the definition of $\II^0$ we have that $\lambda_j(X(\sigma^{k}))$ converges, in $k$, to a positive number whenever $j \in \II^0$.  We have proved~\ref{itm:eigQrate1}.  

Now let $j \in \II^i$ with $i \in \{1,\dotso, \sd(\F)\}$.  By Lemma~\ref{lem:sturmeigbounds} we have, 
\begin{equation}
\label{eq:eigQrate1}
\lambda_j(X(\sigma^k)) = \OO \left( \left(\sigma^k \right)^{\xi(i)} \right), \ \forall k \in \NN.
\end{equation}
Thus there exists $M > 0$ such that $\lambda_j(X(\sigma^k)) \le M\sigma^{k\xi(i)}$.  Now the sequences $\{ \lambda_j(X(\sigma^k)) \}_{k \in \NN}$ and $\{M \sigma^{k\xi(i)} \}_{k \in \NN}$ satisfy the assumptions of Lemma~\ref{lem:2sequence}.  Therefore,
\[
\liminf_{k \to \infty} \frac{\lambda_j(X(\sigma^{k+1}))}{\lambda_j(X(\sigma^k))} \le \limsup_{k \to \infty} \frac{M \sigma^{(k+1)\xi(i)}}{M\sigma^{k\xi(i)}} = \sigma^{\xi(i)}.
\]
Lastly $\sigma^{\xi(d)} < 1$ holds since, $\sigma \in (0,1)$ and $\xi(i) >0$.
\end{proof}
The following corollary emphasizes how Theorem~\ref{thm:eigQrate} may be used to upper bound the rank of $\bar X$.
\begin{cor}
\label{corr:eigQrate1}
Let ${\{\Xa : \alpha > 0\}}$ be a central path satisfying Assumption~\ref{assump:para} for a spectrahedron $\F=\F(\A,b)$ that satisfies Assumption~\ref{assump:F}.  Assume further that ${\epsilon^b(\Xa, \F) = \OO(\alpha)}$.  Let~$r$ denote the maximum rank over $\F$ and let ${\sigma\in (0,1)}$.  Then,
\[
\liminf_{k\rightarrow \infty} \frac{\lambda_i (X(\sigma^{k+1}))}{\lambda_i(X(\sigma^k))} \le \sigma^{\xi(1)}= \sigma^{2^{-(\sd(\F)-1)}}<1 \ \iff \ i > r.
\]
\end{cor}
The number $\sigma^{2^{-(\sd(\F)-1)}}$ serves as a threshold so that limit inferiors of the eigenvalue ratios lie below this number if, and only if, those eigenvalues converge to $0$.  For large singularity degree, it may be difficult to distinguish $\sigma^{2^{-(\sd(\F)-1)}}$ from $1$, numerically.  However if we can identify another number, say $\tau \in (0,1)$, that is numerically distinguishable from $1$ and there exists a positive integer $\overline r$ such that, 
\[
\liminf_{k\rightarrow \infty} \frac{\lambda_i (X(\sigma^{k+1}))}{\lambda_i(X(\sigma^k))} \le \tau \ \iff \ i > \overline r,
\]
then $\overline r$ is an upper bound on the maximum rank, $r$, over $\F$.  We state this result formally in the following.
\begin{cor}
\label{corr:eigQratebound}
Let ${\{\Xa : \alpha > 0\}}$ be a central path satisfying Assumption~\ref{assump:para} for a spectrahedron $\F=\F(\A,b)$ that satisfies Assumption~\ref{assump:F}.  Assume further that ${\epsilon^b(\Xa, \F) = \OO(\alpha)}$.  Let~$r$ denote the maximum rank over $\F$ and let ${\sigma \in (0,1)}$.  Suppose there exists ${\tau \in (0,1)}$ and~${\overline r \in \{1,\dotso,n\}}$ such that
\[
\liminf_{k\rightarrow \infty} \frac{\lambda_i (X(\sigma^{k+1}))}{\lambda_i(X(\sigma^k))} \le \tau \ \iff \ i > \overline r.
\]
Then $\overline r \ge r$.
\end{cor}

\begin{remark}
In the results of this section we require knowledge of the magnitude of backward error for $\F$.  When the SDP in \eqref{eq:sdp} has a non-trivial objective, we cannot expect to know the optimal value $p^*$.  Consequently, backward error is intractable.  However, primal dual algorithms measure the duality gap which is an overestimate for $( \langle C, X \rangle - p^*  )$.  Thus,
\[
\epsilon^b(X, \F) \le \dist \left(X, \{S : \widehat{\A}(S) = \widehat b\} \right) + \dist(X, \Snp) + ( \langle C, X \rangle - p^* ).
\] 
In particular, for the central path $\{\Xa : \alpha > 0\}$ it holds that,
\[
\dist \left(\Xa, \{S : \widehat{\A}(S) = \widehat b\} \right) + \dist(\Xa, \Snp) + ( \langle C, \Xa \rangle - p^* ) = \OO(\alpha) \ \implies \ \epsilon^b(\Xa, \F) = \OO(\alpha),
\]
the assumed bound in the above results.
\end{remark}

\subsection{Bounds on Forward Error and Singularity Degree}
\label{sec:boundferror}

Any bound on maximum rank, such as the one from the previous section, may be used to provide lower bounds on forward error and singularity degree.
\begin{theorem}
\label{thm:ferrorbound}
Let ${\{\Xa : \alpha > 0\}}$ be a central path satisfying Assumption~\ref{assump:para} for a spectrahedron $\F=\F(\A,b)$ that satisfies Assumption~\ref{assump:F}.  If $\overline r$ is an upper bound for the maximum rank over~$\F$ then for all $\alpha > 0$ it holds that,
\[
\epsilon^f(X(\alpha), \F)  \ge \lVert \begin{pmatrix}
\lambda_{\overline r+1}(\Xa) & \cdots & \lambda_n(\Xa) 
\end{pmatrix}^T \rVert_2.
\]
\end{theorem}
\begin{proof}
Let $\overline r$ be as in the hypothesis and let $\alpha > 0$.  Since $\F$ is a closed convex set, there exists $X\in \F$ such that,
\[
\epsilon^f(X(\alpha), \F) = \lVert \Xa - X\rVert_F.
\]
Then observing that $\lVert S \rVert_F = \lVert \lambda(S) \rVert_2$ for any $S \in \Sn$ we have,
\begin{align*}
\epsilon^f(X(\alpha), \F)^2 &= \lVert \Xa - X\rVert_F^2 \\
&= \lVert \Xa \rVert_F^2 + \lVert X \rVert_F^2 - 2 \langle \Xa, X \rangle \\
 &= \lVert \lambda(\Xa) \rVert_2^2 + \lVert \lambda(X) \rVert_2^2 - 2 \langle \Xa, X \rangle.
\end{align*}
Applying the classical bound $\langle X,Y \rangle \le \lambda(X)^T \lambda(Y)$,~e.g.~\cite{Fan:50,hw53}, we get,
\begin{align*}
\epsilon^f(X(\alpha), \F)^2 &\ge \lVert \lambda(\Xa) \rVert_2^2 + \lVert \lambda(X) \rVert_2^2 - 2 \lambda(\Xa)^T \lambda(X)\\
&= \lVert \lambda(\Xa) - \lambda(X) \rVert_2^2 \\
&\ge \lVert \begin{pmatrix}
\lambda_{\overline r+1}(\Xa) & \cdots & \lambda_n(\Xa) 
\end{pmatrix}^T \rVert_2^2.
\end{align*}
Taking the square root of both sides yields the desired result.
\end{proof}

\begin{theorem}
\label{thm:lboundsd}
Let ${\{\Xa : \alpha > 0\}}$ be a central path satisfying Assumption~\ref{assump:para} for a spectrahedron $\F=\F(\A,b)$ that satisfies Assumption~\ref{assump:F}.  Assume further that ${\epsilon^b(\Xa, \F) = \OO(\alpha)}$.  Let~$\overline r$ be an upper bound on the maximum rank over $\F$ and let ${\sigma \in (0,1)}$.  Suppose $\underline d$ is the smallest positive integer such that,
\begin{equation}
\label{eq:1lboundsd}
\liminf_{k\rightarrow \infty} \frac{\lambda_i (X(\sigma^{k+1}))}{\lambda_i(X(\sigma^k))} \le \sigma^{2^{-(\underline d - 1)}} \ \iff \ i > \overline r.
\end{equation}
Then $\underline d \le \sd(\F)$.
\end{theorem}
\begin{proof}
Let $r$ denote the maximum rank over $\F$.  Suppose $\overline r > r$.  Then by Corollary~\ref{corr:eigQrate1} and \eqref{eq:1lboundsd} we have,
\[
\sigma^{2^{-(\underline d-1)}} <\liminf_{k\rightarrow \infty} \frac{\lambda_{r+1} (X(\sigma^{k+1}))}{\lambda_{r+1}(X(\sigma^k))} \le \sigma^{2^{-(\sd(\F) - 1)}}.
\]
It follows that $\underline d \le \sd(\F)$.  Now suppose that $\overline r = r$.  Then by Corollary~\ref{corr:eigQrate1} we have,
\begin{equation}
\label{eq:2lboundsd}
\liminf_{k\rightarrow \infty} \frac{\lambda_i (X(\sigma^{k+1}))}{\lambda_i(X(\sigma^k))} \le \sigma^{2^{-(\sd(\F) - 1)}} \ \iff \ i > r = \overline r.
\end{equation}
Out of all positive integers that could replace $\sd(\F)$ in \eqref{eq:2lboundsd}, we chose $\underline d$ to be the smallest.  Hence ${\underline d \le \sd(\F)}$, as desired.
\end{proof}

\section{Singularity Degree as a Measure of Hardness}
\label{sec:boundsd}
It is certainly possible to construct a parametric curve $\{ \Xa : \alpha > 0\}$ with the properties of Assumption~\ref{assump:para}, for which singularity degree is large, but forward error is small.  For instance, the path defined as ${\Xa := \bar X + \alpha I}$, where $\bar X \in \relint(\F)$, exhibits fast convergence and low forward error irrespective of the singularity degree.   This demonstrates that singularity degree is not a sufficient condition for slow convergence for \emph{all} parametric curves.  However, for many of the central paths constructed by state of the algorithms, empirical evidence indicates otherwise. In this section we present several results that give credence to the notion that singularity degree is a measure of hardness for a family of central paths.  In Section~\ref{sec:cpaths} we introduce the family of central paths and in Section~\ref{sec:sdhard} we present the main results of the section.

\subsection{Analaysis of a Family of Central Paths}
\label{sec:cpaths}
The classical \emph{interior point} method for SDPs is based on a central path that is constructed by assuming that both the primal and the dual satisfy the Slater condition.  As this assumption is quite restrictive, \emph{infeasible} central paths assuming weaker conditions have been subsequently proposed.  Among these are \cite{NestToddYe:96,int:deklerk7,PotShe:95,KlerkRoosTerlakyinit:97,lusz00}.  

In \cite{PotShe:95}, Potra and Sheng proposed a family of infeasible
central paths that are based on perturbing the feasible region so as to
satisfy the Slater condition, and then decreasing the perturbation. Our
analysis requires a complete knowledge of the algebraic representation
of the spectrahedron.  Thus we assume that $\F = \F(\A,b)$ for a known
linear map $\A$ and vector $b$.  In terms of the SDP of \eqref{eq:sdp},
this equates to feasibility problems, i.e.,~${C=0}$ and $\A:=\widehat \A$, $b:= \widehat b$.  From the family of paths proposed by Potra and Sheng we choose the path $\{\Xa: \alpha > 0\}$ defined by,
\begin{equation}
\label{eq:param}
\begin{cases}
\Xa := \arg \max \{ \alpha \log \det ( X) : X\in \Fa \}, \\
\Fa := \{X \in \Snp : \A(X) = \ba \}, \\
\ba := b+ \alpha \A(B),
\end{cases}
\end{equation}
where $B \succ 0$ is fixed.  The matrix $\Xa$ exists for each $\alpha > 0$ if, and only if, $\Fa$ is nonempty and bounded.  Hence the following assumption.
\begin{assump}
\label{assump:Fpotshe}
Let $\F= \F(\A,b)$ be a spectrahedron defined in terms of a map $\A : \Sn \to \Rm$ and a vector $b \in \Rm$.  We assume that,
\begin{enumerate}[$(i)$]
\item $\F$ is nonempty, bounded, ${\sd(\F) \ge 1}$, and ${\F \ne \{0\}}$,
\item $\A$ is surjective.
\end{enumerate}
\end{assump}
Assumption~\ref{assump:Fpotshe} differs from Assumption~\ref{assump:F} in the additional requirements that $\F$ is bounded and that $\A$ is surjective.  The need for a bounded $\F$ has already been discussed and the restriction on $\A$ ensures a unique $y \in \Rm$ for every $Z \in \range(\A^*)$, a property that will prove convenient in the subsequent discussion.

It is easy to see that if $\F \ne \emptyset$ then $\Fa$ has a Slater point for every $\alpha >0$.  For instance, the set~${\F + \alpha B}$ has positive definite elements and is contained in $\Fa$.  Since $\Xa$ is chosen to be the determinant maximizer over $\Fa$ it follows that ${\Xa \succ 0}$ and $\Xa \in \relint(\F)$ for each $\alpha >0$.  We have thus shown that this central path satisfies Assumption~\ref{assump:para}~\ref{itm:2para}.  In the remainder of this section we show that it also possesses the other properties of Assumption~\ref{assump:para}.  Namely, that $\{\Xa:\alpha>0\}$ is smooth and converges to a matrix in $\relint(\F)$ as ${\alpha \searrow 0}$. 

The optimality conditions for \eqref{eq:param} yield the primal-dual central path,
\begin{equation}
\label{eq:optimalsystem}
 \left \{ (\Xa, \ya, \Za) \in \Snpp \times \Rm \times \Snpp : \begin{bmatrix}
 \A^*(\ya)-\Za  \\
\A(\Xa) - \ba \\
\Za \Xa - \alpha I
\end{bmatrix} = 0, \ \alpha > 0 \right \}.
\end{equation}
This primal-dual central path is smooth around every $\alpha > 0$ and admits the following convergence result.
 \begin{theorem}
 \label{thm:paramconv}
Let ${\{(\Xa, \ya, \Za): \alpha > 0\}}$ be the primal-dual central path of \eqref{eq:optimalsystem} for a spectrahedron $\F = \F(\A,b)$ satisfying Assumption~\ref{assump:Fpotshe}. Then,
\[ 
\lim_{\alpha \searrow 0} (\Xa, \ya , \Za) = (\bar X, \bar y, \bar Z) \in \Snp \times \Rm \times \Snp,
\]
with 
\begin{equation}
\label{eq:1paramconv}
\begin{cases}
\bar X \in \relint(\F), \\
\bar Z = \A^*(\bar y), \\
\bar Z \in \relint \{ Z \in \Snp \setminus \{0\} : Z = \A^*(y), \ y^Tb = 0, \ y \in \Rm\}.
\end{cases}
\end{equation}
\end{theorem}
\begin{proof}
While we are not aware of this exact result in the literature, it may be deduced from two known results.  Firstly, in~\cite{int:Goldfarb15} the authors show that every limit point of the primal-dual central path exhibits the properties of \eqref{eq:1paramconv}.  Then a lemma of Milnor~\cite{mi68} may be used as in~\cite{Halicka:01,HalickaKlerkRoos:01} to show that the set of limit points is a singleton.
\end{proof}

An immediate consequence of Theorem~\ref{thm:paramconv} is a statement about the convergence rates of eigenvalues of $\Xa$ and $\Za$.
\begin{cor}
\label{cor:eigrates}
Let ${\{(\Xa, \ya, \Za): \alpha > 0\}}$ be the primal-dual central path of \eqref{eq:optimalsystem} for a spectrahedron $\F = \F(\A,b)$ satisfying Assumption~\ref{assump:Fpotshe}.  Let ${(\bar X, \bar y, \bar Z)}$ be the limit point of the primal-dual central path and let $r$ and $q$ denote the rank of $\bar X$ and $\bar Z$, respectively.  Let ${\bar \alpha >0}$.  Then for every $\alpha \in (0,\bar \alpha)$ we have,
\begin{equation}
\label{eq:coreigratesX}
\lambda_i(\Xa) = \begin{cases}
\Theta(1) \quad &i \le r, \\
\Omega(\alpha) \text{ and } \ne \OO(\alpha) &i \in [r+1,n-q], \\
\Theta(\alpha) &i\ge n-q+1,
\end{cases}
\end{equation}
and
\begin{equation}
\label{eq:coreigratesZ}
\lambda_i(\Za) = \begin{cases}
\Theta(1) \quad &i \le q, \\
\OO(1) \text{ and } \ne \Omega(1) &i \in [q+1,n-r], \\
\Theta(\alpha) &i\ge n-r+1.
\end{cases}
\end{equation}
\end{cor}
\begin{proof}
The convergence result of Theorem~\ref{thm:paramconv} implies that the $r$ largest eigenvalues of $\Xa$ and the $q$ largest eigenvalues of $\Za$ converge to positive values, hence are $\Theta(1)$.  Moreover, the relation $\Za = \alpha \Xa^{-1}$ gives that the $q$ smallest eigenvalues of $\Xa$ and the $r$ smallest eigenvalues of $\Za$ converge to $0$ at a rate that is $\Theta(\alpha)$.   

Now let $i \in \{r+1,\dotso,n-q\}$.  The lower bound $\Omega(\alpha)$ holds since, 
\[
\lambda_i(\Xa) \ge \lambda_n(\Xa) = \Theta(\alpha).
\] 
Now suppose, for the sake of contradiction, that~${\lambda_i(\Xa) =\OO(\alpha)}$.  Combining with the lower bound, $\Omega(\alpha)$, we conclude that~${\lambda_i(\Xa) = \Theta(\alpha)}$.  Once again using the relation $\Za = \alpha \Xa^{-1}$ implies that there are $q+1$ eigenvalues of $\Za$ that are bounded away from $0$, contradicting the assumption that $\rank(\bar Z) = q$ and the statement of Theorem~\ref{thm:paramconv}:
\[
\bar Z \in \relint \{ Z \in \Snp \setminus \{0\} : Z = \A^*(y), \ y^Tb = 0, \ y \in \Rm\}.
\]
The remaining bounds on $\Za$ are obtained in similar fashion.
\end{proof}

\subsection{On Singularity Degree and Slow Convergence}
\label{sec:sdhard}
An immediate implication of Corollary~\ref{cor:eigrates} is that fast convergence of eigenvalues does not occur when singularity degree is greater than $1$.
\begin{theorem}
\label{thm:sd1}
Let ${\{ \Xa : \alpha > 0\}}$ be the central path of \eqref{eq:param} for a spectrahedron $\F = \F(\A,b)$ satisfying Assumption~\ref{assump:Fpotshe}.  If $\sd(\F) > 1$, there exists an eigenvalue of $\Xa$ that converges to $0$ at a rate that is not $\OO(\alpha)$.
\end{theorem}
For the second theorem of this section we need the following technical results.
\begin{lemma}
\label{lem:yaTb}
Let ${\{ (\Xa,\ya,\Za) : \alpha > 0\}}$ be the primal-dual central path of \eqref{eq:optimalsystem} for a spectrahedron $\F = \F(\A,b)$ satisfying Assumption~\ref{assump:Fpotshe}. Suppose that $b \ne 0$ and let ${v^1,\dotso,v^{m-1} \in \Rm}$ form a basis for $b^{\perp}$.  Let $\bar \alpha > 0$ and for all $\alpha \in (0,\bar \alpha)$ let ${\beta(\alpha)}$ and ${\nu_1(\alpha), \dotso, \nu_{m-1}(\alpha)}$ be real coefficients such that,
\[
\ya = \beta(\alpha)b + \sum_{i=1}^{m-1} \nu_i(\alpha) v^i. 
\]
Then for all $\alpha \in (0,\bar \alpha)$,
\[
\lvert \beta(\alpha) \rvert = \Theta(\alpha).
\]
\end{lemma}
\begin{proof}
By definition of $b(\alpha)$ in \eqref{eq:param} and by \eqref{eq:optimalsystem} we have,
\begin{align*}
 y(\alpha)^Tb  &=  y(\alpha)^T(b(\alpha) - \alpha \A(B))  \\
&=  y(\alpha)^T(\A(X(\alpha)) - \alpha \A(I))  \\
&=  \langle \A^*(y(\alpha)), X(\alpha) \rangle - \alpha \langle B, \A^*(y(\alpha)) \rangle  \\
&= \alpha \langle X(\alpha)^{-1}, X(\alpha) \rangle - \alpha \langle B, \A^*(y(\alpha)) \rangle \\
&= \alpha (n - \langle B, \A^*(y(\alpha)) \rangle).
\end{align*}
Now $\langle B, \A^*(y(\alpha)) \rangle$ is bounded below and above on $\alpha \in (0,\bar \alpha)$, implying that~${\lvert y(\alpha)^Tb \rvert = \Theta(\alpha)}$.  On the other hand, ${\{v^1, v^2, \dotso, v^{m-1} \} \in b^{\perp}}$ by construction.  Therefore ${y(\alpha)^Tb = \beta(\alpha)\lVert b \rVert^2}$, yielding,
\[
\lvert \beta(\alpha) \rvert = \frac{\lvert y(\alpha)^Tb \rvert}{\lVert b \rVert^2} = \Theta(\alpha),
\]
as desired.
\end{proof}
\begin{lemma}
\label{lem:Zprincipal}
Let ${\{ (\Xa,\ya,\Za) : \alpha > 0\}}$ be the primal-dual central path of \eqref{eq:optimalsystem} for a spectrahedron $\F = \F(\A,b)$ satisfying Assumption~\ref{assump:Fpotshe}.  Let $\hat Z(\alpha) \in \Sc^{\hat n}$ be a principal submatrix of $\Za$, for some $\hat n \in [r+1,n-1]$.  Then for all $\alpha \in (0, \bar \alpha)$ for some fixed $\bar \alpha > 0$ it holds that,
\[
\lVert \hat Z(\alpha) \rVert = \Omega \left(\alpha^{1-2^{-(\sd(\F)-1)}} \right) \ne \OO(\alpha).
\]
\end{lemma}
\begin{proof}
By interlacing eigenvalues and Corollary~\ref{cor:eigrates} we have,
\[
\lVert \hat Z(\alpha) \rVert_2 = \lambda_1(\hat Z(\alpha)) \ge \lambda_{n-r}(\Za) = \Omega \left(\alpha^{1-2^{-(\sd(\F)-1)}} \right).
\]
The bound also holds for any other norm on $\Sn$. 
\end{proof}
Now we show that after a suitable orthogonal transformation, $\Za$ admits a block partition where at least $\sd(\F)$ of these blocks converge to $0$, each at a different rate. 
\begin{theorem}
\label{thm:sddetect}
Let ${\{ (\Xa,\ya,\Za) : \alpha > 0\}}$ be the primal-dual central path of \eqref{eq:optimalsystem} for a spectrahedron $\F = \F(\A,b)$ satisfying Assumption~\ref{assump:Fpotshe}.  Let $\bar \alpha >0$ be fixed.  Then there exists an integer ${d \in [\sd(\F),\bar m]}$ where,
\[
\bar m = \begin{cases}
m \quad &\text{if } b=0, \\
m-1 &\text{otherwise},
\end{cases}
\]  
and a suitable orthogonal transformation of $\F$, such that,
\[
\Za = \begin{bmatrix}
Z_{d+1}(\alpha) & * & \cdots & * \\
* & Z_d(\alpha) & \cdots & * \\
\vdots & \vdots & \ddots & \vdots \\
*  & * & \cdots & Z_1(\alpha)
\end{bmatrix},
\]
where for all $\alpha \in (0,\bar \alpha)$ it holds that,
\begin{enumerate}[$(i)$]
\item $Z_1(\alpha) \to S_1 \succ 0$ and $Z_i(\alpha) \to 0$ for all $i \in \{2,\dotso, d+1\}$,
\label{itm:1sddetect}
\item $\frac{\lambda_{\min}(Z_i(\alpha))}{\lambda_{\max}(Z_{i+1}(\alpha))} \to \infty$ for all $i \in \{1,\dotso,d\}$,
\label{itm:2sddetect}
\item $\lambda_{\min}(Z_i(\alpha)) = \Theta \left(\lambda_{\max}(Z_i(\alpha)) \right)$ for all $i \in \{1,\dotso,d+1\}$,
\label{itm:3sddetect}
\item $\lVert Z_{d+1}(\alpha)\rVert = \Theta(\alpha)$.
\label{itm:4sddetect}
\end{enumerate}
\end{theorem}
\begin{proof}
We assume, without loss of generality, that,
\begin{equation}
\label{eq:faceFsddetect}
\face(\F) = \begin{bmatrix}
\Srp & 0 \\
0 & 0
\end{bmatrix}.
\end{equation}

As above, let $(\bar X, \bar y, \bar Z)$ be the limit point of the primal-dual central path.  We know that $\bar Z$ is an exposing vector for a face containing $\face(\F)$.  Thus ${\bar y^Tb = 0}$.  Without loss of generality we may assume that,
\begin{equation}
\label{eq:0sddetect}
\bar Z =: \begin{bmatrix}
0 & 0 \\
0 & S_1
\end{bmatrix}, \ S_1 \succ 0.
\end{equation}
Now we define $y^1 := \bar y$ and choose ${v^1,\dotso,v^{m_v} \in \spanl \{b,y^1\}^{\perp}}$ for some ${m_v \le m}$ so that the collection of vectors~$\{y^1,v^1, \dotso, v^{m_v}\}$ is a basis for $b^{\perp}$.  Note that when $b\ne 0$, then~$\{b,y^1,v^1, \dotso, v^{m_v}\}$ is also a basis for $\Rm$.  

Now for each ${\alpha > 0}$ there exist coefficients $\beta(\alpha)$ and ${\nu_1(\alpha), \dotso, \nu_{m_v}(\alpha)}$ and $\gamma_1(\alpha)$ such that,
\begin{equation}
\label{eq:1sddetect}
y(\alpha) = \gamma_1(\alpha)y^1 + \beta(\alpha)b + \sum_{i=1}^{m_v} \nu_i(\alpha)v^i.
\end{equation}
Since $\ya \to y^1$ we have $\gamma_1(\alpha) \to 1$ and $\gamma_1(\alpha)$ dominates the other coefficients.  That is,
\begin{equation}
\label{eq:2sddetect}
\lim_{\alpha \searrow 0} \frac{\sum_{i=1}^{m_v} \lvert \nu_i(\alpha) \rvert }{\gamma_1(\alpha)} = 0.
\end{equation}
Now let us consider a block partition of $\Za$ according to the block partition of $\bar Z$ in \eqref{eq:0sddetect}.  We have,
\begin{equation}
\label{eq:3sddetect}
Z(\alpha) = \begin{bmatrix}
0 & 0 \\
0 & \gamma_1(\alpha)S_1
\end{bmatrix} + \A^* \left(\beta(\alpha)b + \sum_{i=1}^{m_v} \nu_i(\alpha)v^i \right).
\end{equation}
Note that the two diagonal blocks of $\Za$ in \eqref{eq:3sddetect} possess properties~\ref{itm:1sddetect} and~\ref{itm:2sddetect}. 

Let $Z_{11}(\alpha)$ denote the upper left block of $\Za$.  We consider two possibilities.  First, suppose that $\lVert Z_{11}(\alpha) \rVert = \OO(\alpha)$.  In this case these two diagonal blocks also satisfy properties~\ref{itm:3sddetect} and~\ref{itm:4sddetect} and we let $d=1$.  That $d \le \bar m$ is easy to see.  Next, Lemma~\ref{lem:Zprincipal} implies that $Z_{11}(\alpha)$ has at most $r$ rows.  Moreover, by our assumption on the facial structure of $\face(\F)$ in \eqref{eq:faceFsddetect} we conclude that $Z_{11}(\alpha)$ has exactly $r$ rows, otherwise $\bar Z$ exposes a face that is strictly smaller than $\face(\F)$.  Thus we have $\sd(\F) = 1 \le d$, as desired.

The second possibility is that $\lVert Z_{11}(\alpha) \rVert \ne \OO(\alpha)$.   In this case, at least one of the coefficients other than $\gamma_1(\alpha)$ converges to $0$ at a rate not equal to $\OO(\alpha)$.  This coefficient is not $\beta(\alpha)$, since Lemma~\ref{lem:yaTb} implies that ${\beta(\alpha) = \Theta(\alpha)}$ when $b \ne 0$.  When $b=0$ we may set $\beta(\alpha) = 0$ as it is irrelevant.  Thus we conclude that $\lvert \nu_i(\alpha) \rvert \ne \OO(\alpha)$ for some $i \in \{1,\dotso,m_v\}$.

Now let, $y^2$ be a limit point of $\beta(\alpha)b + \sum_{i=1}^{m_v} \nu_i(\alpha)v^i$, after normalizing.  By the arguments above, ${y^2 \in \spanl \{v^1,\dotso, v^{m_v}\}}$ and thus ${(y^2)^Tb = 0}$.  Secondly, $(\A^*(\beta(\alpha)b + \sum_{i=1}^{m_v} \nu_i(\alpha)v^i) )_{11}$ is positive definite for every $\alpha>0$ by \eqref{eq:3sddetect}. This implies that ${(\A^*(y^2))_{11} \succeq 0}$.  Thus if $(\A^*(y^2))_{11}$ is not the zero matrix it is an exposing vector in the second step of facial reduction.  

Let us first address the case ${(\A^*(y^2))_{11} = 0}$.  Here we let $w^1 := y^2$ and choose 
\[
v^1, \dotso, v^{m_v} \in \spanl \{b,y^1,w^1\}^{\perp}
\] 
for some $m_v$, different than the previously used $m_v$, so that ${\{y^1,w^1, v^1, \dotso, v^{m_v}\}}$ is a basis for~$b^{\perp}$.  Now we repeat the above process until we obtain a new $y^2$ such that ${(\A^*(y^2))_{11} \ne 0}$,~i.e.,~we are in the second case.

Now we may assume that we have obtained $y^2$ as above and ${(\A^*(y^2))_{11} \ne 0}$.  We also assume, without loss of generality that,
\begin{equation}
\label{eq:4sddetect}
(\A^*(y^2))_{11} = \begin{bmatrix}
0 & 0 \\
0 & S_2
\end{bmatrix}, \ S_2 \succ 0.
\end{equation}
Then the matrix,
\begin{equation}
\label{eq:5sddetect}
\begin{bmatrix}
0 & 0 & 0 \\
0 & S_2 & 0 \\
0 & 0 & S_1
\end{bmatrix},
\end{equation}
exposes a face containing $\face(\F)$ and this face is smaller than the one exposed by $\bar Z$.  In other, words, we have obtained a better exposing vector.  Let us assume that we have accumulated $m_w$ vectors of the type $w^1$ obtained in the case ${(\A^*(y^2))_{11} = 0}$.  Then we choose,
\begin{equation}
\label{eq:5asddetect}
v^1,\dotso, v^{m_v} \in \spanl \{b,y^1,y^2,w^1,\dotso,w^{m_w} \}^{\perp},
\end{equation}
so that $\{y^1,y^2,w^1,\dotso,w^{m_w},v^1,\dotso, v^{m_v}\}$ is a basis for $b^{\perp}$.
As above, there exist coefficients, 
\begin{equation}
\label{eq:5csddetect}
\beta(\alpha), \ \gamma_1(\alpha), \gamma_2(\alpha), \ \omega_1(\alpha),\dotso,\omega_{m_w}(\alpha), \ \nu_1(\alpha), \dotso, \nu_{m_v}(\alpha),
\end{equation}
such that,
\[
\ya = \beta(\alpha)b + \sum_{i=1}^{2}\gamma_i(\alpha)y^i + \sum_{i=1}^{m_w} \omega_i(\alpha)w^i + \sum_{i=1}^{m_v}\nu_i(\alpha)v^i.
\]
Then,
\begin{equation}
\Za = \begin{bmatrix}
0 & 0 & 0 \\
0 & \gamma_2(\alpha)S_2 & 0 \\
0 & 0 & \gamma_1(\alpha) S_1
\end{bmatrix} + \A^* \left( \beta(\alpha)b+\sum_{i=1}^{m_v}\nu_i(\alpha)v^i+\sum_{i=1}^{m_w} \omega_i(\alpha)w^i\right),
\end{equation}
where the upper left block is $\A^*_{11} \left(\beta(\alpha)b+\sum_{i=1}^{m_v}\nu_i(\alpha)v^i\right)$.
By construction we have
\begin{equation}
\label{eq:6sddetect}
\frac{\gamma_1(\alpha)}{\gamma_2(\alpha)} \to \infty \text{ and } \frac{\gamma_2(\alpha)}{b(\alpha) + \sum_{i=1}^{m_v} \lvert \nu_i(\alpha) \rvert} \to \infty.
\end{equation}
Thus we conclude that the diagonal blocks of $\Za$ satisfy properties~\ref{itm:1sddetect} and~\ref{itm:2sddetect}.  In addition, the blocks containing $\gamma_1(S_1)$ and $\gamma_2(S_2)$ satisfy property~\ref{itm:3sddetect}.

By Lemma~\ref{lem:Zprincipal} we may continue in this fashion until,
\begin{equation}
\label{eq:7sddetect}
\Za = \begin{bmatrix}
0 & 0 & \cdots & 0 \\
0 & \gamma_d(\alpha)S_d & \cdots & 0 \\
\vdots & \vdots & \ddots & \vdots \\
0 & 0 & \cdots & \gamma_1(\alpha)S_1
\end{bmatrix} + \A^* \left( \beta(\alpha)b+\sum_{i=1}^{m_v}\nu_i(\alpha)v^i+\sum_{i=1}^{m_w} \omega_i(\alpha)w^i\right),
\end{equation}
for some positive integer $d$ and the upper left block has norm that is $\OO(\alpha)$.  By reasoning analogous to that of the discussion following \eqref{eq:3sddetect}, we conclude that the blocks of $\Za$, according to the block partition of \eqref{eq:7sddetect}, satisfy properties~\ref{itm:1sddetect} - \ref{itm:4sddetect} and that $d \in [\sd(\F), \bar m]$, as desired.
\end{proof}

The issue with different rates of convergence among the blocks of $\Za$ that vanish, is one of a practical nature.  Suppose a path-following algorithm is applied and accurately follows the primal-dual central path.  Then, once the the block of $\Za$ that converges to $0$ at the fastest rate, $Z_{d+1}(\alpha)$, reaches machine precision, the remaining blocks cannot be made smaller.  Hence the forward error is a function of the difference between the rate of convergence of the slowest block, $Z_2(\alpha)$, and the fastest block, $Z_{d+1}(\alpha)$.

The integer $d$ of Theorem~\ref{thm:sddetect} actually provides an upper bound on $\sd(\F)$ that complements the lower bound of Theorem~\ref{thm:lboundsd}.  However, this upper bound is generally intractable due to its reliance on an unknown orthogonal transformation.  If the statement of the theorem can be translated to a statement about convergence rates of blocks of eigenvalues, then we would have a tractable upper bound on $\sd(\F)$.  However this may not be true as illustrated by the parametric sequence,
\begin{equation}
\label{eq:cexample}
S(\alpha) := \begin{bmatrix}
3 & \alpha^{1/2} & 0 \\
\alpha^{1/2} & \frac{\alpha}{3-\alpha^2} & 0 \\
0 & 0 & \alpha^3
\end{bmatrix}, \ \alpha  > 0.
\end{equation}
Here $S(\alpha)$ has different rates of convergence among the diagonal, but the two eigenvalues that vanish, do so at the same rate, $\Theta(\alpha^3)$.  One way to guarantee that the diagonal blocks correspond to blocks of eigenvalues is the following.
\begin{corollary}
\label{corr:blocktoeig}
Let ${\{ (\Xa,\ya,\Za) : \alpha > 0\}}$ be the primal-dual central path of \eqref{eq:optimalsystem} for a spectrahedron $\F = \F(\A,b)$ satisfying Assumption~\ref{assump:Fpotshe}.  Let $\bar \alpha > 0$ be fixed and let $d$ be as in Theorem~\ref{thm:sddetect}.  For every $\alpha \in (0, \bar \alpha)$ we assume that $\Za$ has the block structure of Theorem~\ref{thm:sddetect}.  If every principal submatrix of $\Za$ of the form,
\[
S_i(\alpha) = \begin{bmatrix}
Z_i(\alpha) & \cdots & * \\
 \vdots & \ddots & \vdots \\
 * & \cdots & Z_1(\alpha)
\end{bmatrix}, \ i \in \{2,\dotso,d\},
\] 
satisfies $\lambda_{\min}(S_i(\alpha)) = \Theta \left(\lambda_{\min}(Z_i(\alpha)) \right)$, then there are exactly $d$ different rates of convergence among the eigenvalues of $\Xa$ that vanish.
\end{corollary}
\begin{proof}
Applying the interlacing eigenvalue theorem with the principal submatrices,
\[
\begin{bmatrix}
Z_{d+1}(\alpha) & \cdots & * \\
 \vdots & \ddots & \vdots \\
 * & \cdots & Z_i(\alpha)
\end{bmatrix} \text{ and }
\begin{bmatrix}
Z_i(\alpha) & \cdots & * \\
 \vdots & \ddots & \vdots \\
 * & \cdots & Z_1(\alpha)
\end{bmatrix},
\]
yields the upper bound of $\lambda_{\max}(Z_i(\alpha))$ and the lower bound $\lambda_{\min}(Z_i(\alpha))$ on a block of $\lambda(\Za)$ having the same size as the number of rows $Z_i(\alpha)$.  Since $\lambda_{\min}(Z_i(\alpha)) = \Theta \left( \lambda_{\max}(Z_i(\alpha)) \right)$, these bounds are the same.  Thus we conclude that for each $i \in \{2,\dotso,d+1\}$ there is a block of eigenvalues that converges to $0$ at the same rate as $Z_i(\alpha)$ does.  The desired result follows from the relation $\Xa = \alpha \Za^{-1}$.
\end{proof}
The challenge for this result is that the hypothesis is unverifiable just as the block structure of $\Za$ is unobservable without knowledge of the appropriate orthogonal transformation.  However, our numerical observations indicate that the conclusion of the corollary holds for instances where singularity degree is known.


\section{Numerical Case Studies}
\label{sec:numerics}
In this section we apply the main results above to obtain bounds on forward error, singularity degree, and maximum rank for several problems from the literature.  Our analysis is focused on problems with larger singularity degree, although, we do study one instance with singularity degree $1$, in order to demonstrate `good' convergence.  For some of the instances, the exact singularity degree is known, allowing us to test the quality of our bounds.  In other instances, the singularity degree is not known and we use our bounds to provide an estimate of it.

In order to study the notion that large singularity degree is sufficient, in some sense, for slow convergence, we consider bounded spectrahedra, satisfying Assumption~\ref{assump:Fpotshe}.  We follow the primal-dual central path of \eqref{eq:optimalsystem} with a path-following algorithm based on the Gauss-Newton search direction, see~\cite{KrMuReVaWo:98,KrukDoanW:10}.

For each problem, we present two plots based on a sequence of the type $\{ \sigma^k\}$, as in Section~\ref{sec:boundrankferror}, where $\sigma =0.6$.  The first is the $Q$-convergence ratio as in Theorem~\ref{thm:eigQrate}.  To make the subsequent discussion less cumbersome, we introduce the notation,
\begin{equation}
\label{eq:RQ}
R_Q(i,k) := \frac{\lambda_i(X(\sigma^{k+1}))}{\lambda_i(X(\sigma^k))}, \ i \in \{1,\dotso,n\}, \ k \ge 1.
\end{equation}
We use this first ratio to bound maximum rank (Corollary~\ref{corr:eigQratebound}), forward error (Theorem~\ref{thm:ferrorbound}), and singularity degree (Theorem~\ref{thm:lboundsd}).

The second plot is the ratio of adjacent eigenvalues as in \eqref{eq:m1},
\begin{equation}
\label{eq:R_N}
R_N (i,k) := \frac{\lambda_i(X(\sigma^{k}))}{\lambda_{i+1}(X(\sigma^k))}, \ i \in \{1,\dotso,n-1\}, \ k \ge 1.
\end{equation}
This ratio is used to upper bound the maximum rank as in Proposition~\ref{prop:eig} and to determine the number of different rates of convergence among eigenvalues that vanish. 

For each value of $k$, we approximate the ratios $R_Q(i,k)$ and $R_N(i,k)$ by following the primal-dual central path $(\Xa,\ya,\Za)$ of \eqref{eq:optimalsystem} to the point where $\alpha = \sigma^k$.   Once the ratios have been obtained for $k$ sufficiently large, around $60$, we generate plots of the ratios against $k$ for each $i$ and obtain bounds on maximum rank, forward error, and singularity degree.  The bounds as well as the true values (when available) are recorded in Table~\ref{table}.

\begin{figure}[h]
\centering
\begin{minipage}{.5\textwidth}
  \centering
  \includegraphics[width=\textwidth]{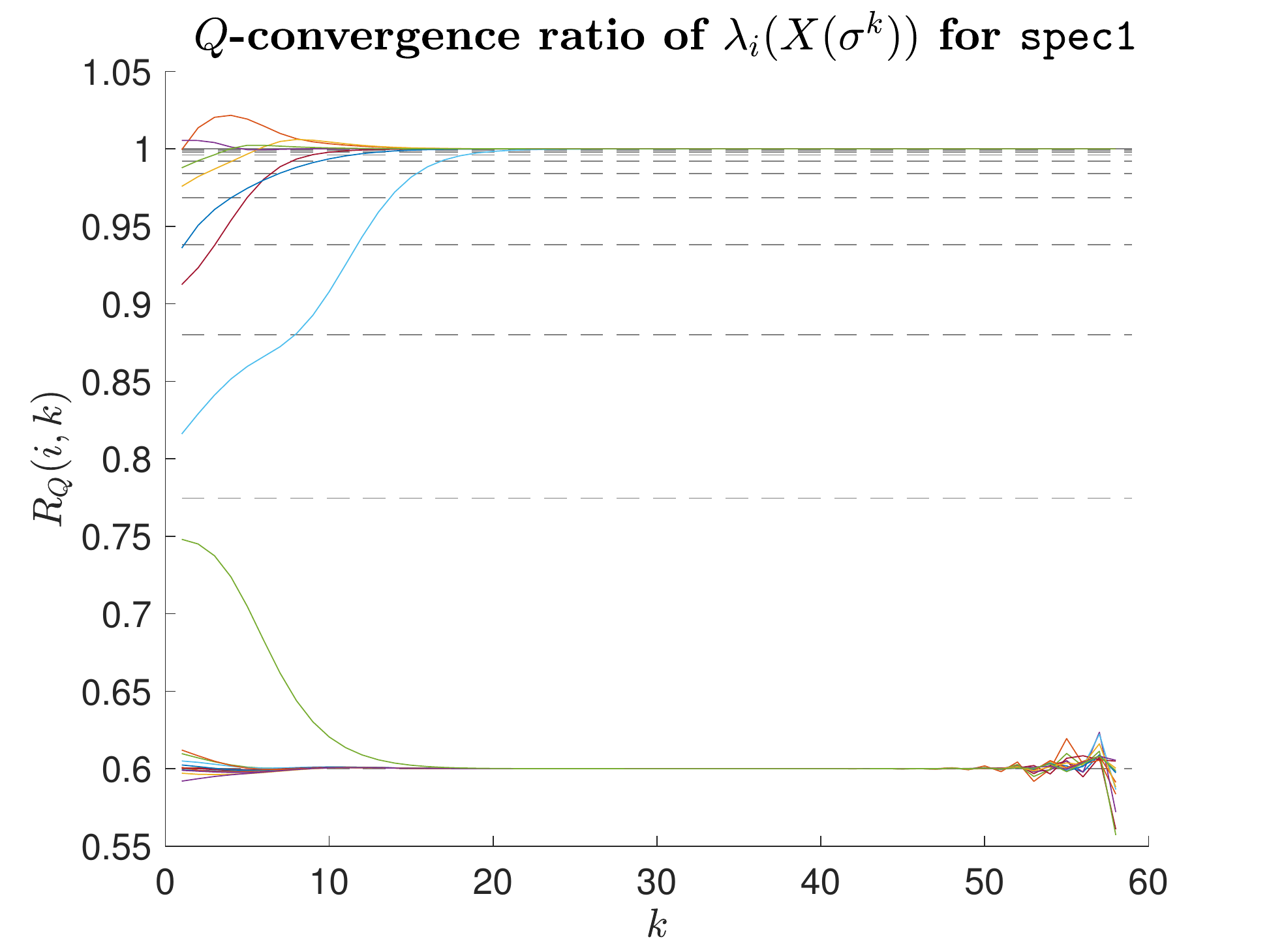}
\end{minipage}%
\begin{minipage}{.5\textwidth}
  \centering
  \includegraphics[width=\textwidth]{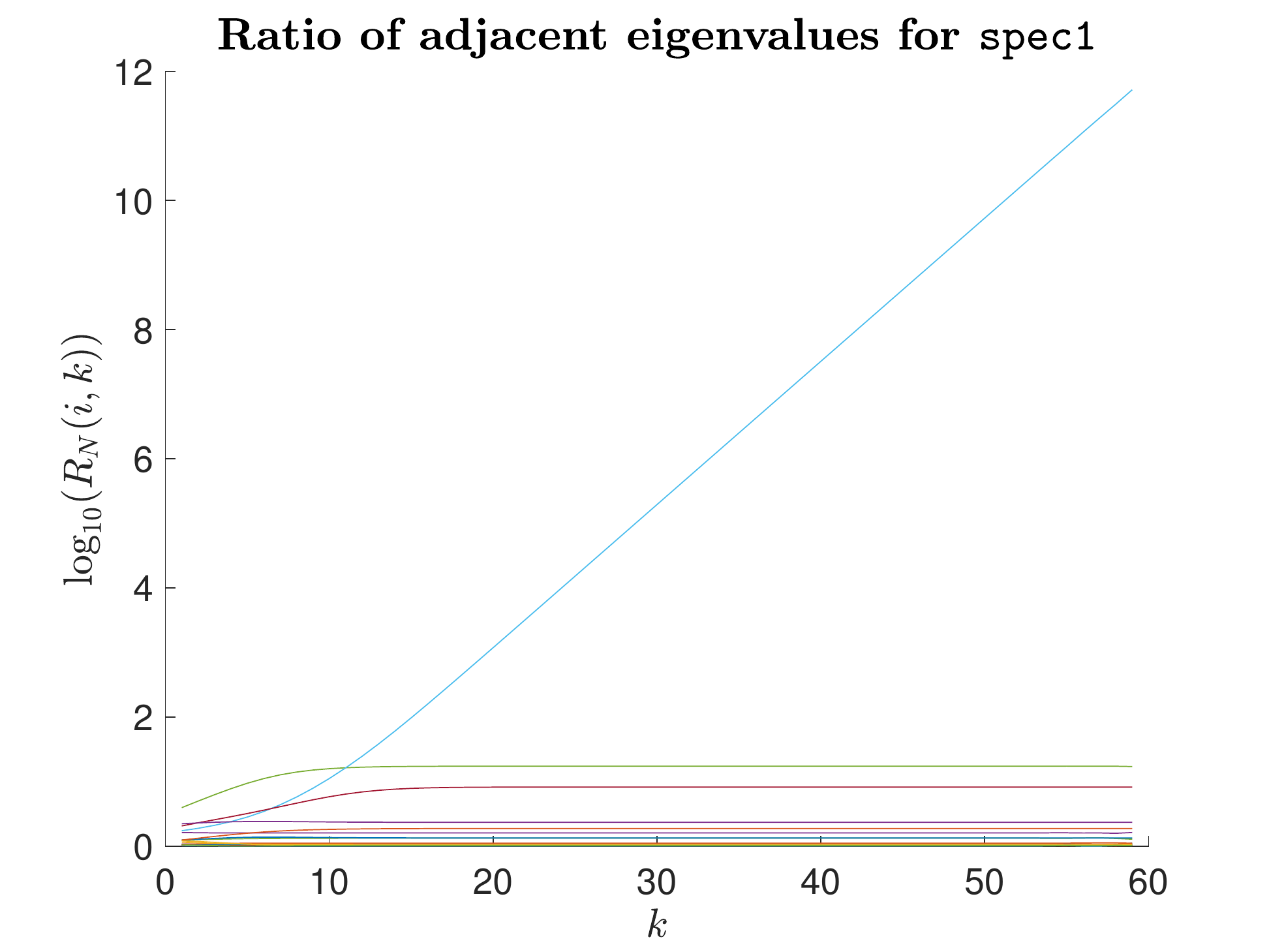}
\end{minipage}
\caption{The dashed lines coincide with the values $\sigma^{2^{-(d-1)}}$ for $d \in \{1,\dotso,n-1\}$.}
\label{fig:test1}
\end{figure}

We denote the first spectrahedron as \texttt{spec1} and it is taken from the class of SDPs introduced in~\cite{WeiWolk:06}.  There, the authors present a method to generate SDPs with specified \emph{complementarity gap}.  By considering the optimal set of one such SDP with strict complementarity, we obtain a spectrahedron with singularity degree $1$.  The problem we consider has size $n=20$ and plots of the ratios $R_Q(i,k)$ and $R_N(i,k)$ are shown in Figure~\ref{fig:test1}.  In the left image, there is a clear distinction between curves that converge to $1$ and curves that do not converge to $1$.  Moreover, if we disregard the irregularity in the last few values of the curves that do not converge to $1$, we may conclude that those curves converge to the smallest dashed line located at $0.6$. This observation, together with Theorem~\ref{thm:sd1}, correctly indicates that the spectrahedron has singularity degree $1$.  Exactly $13$ of the curves converge to $0.6$, yielding $\overline r=7$, the correct approximation of the maximal rank $r$.  The plot on the right side of the figure shows that exactly one curve blows up and it is the curve corresponding to $i=7=r$.  This indicates, as expected, two groups of eigenvalues of~$X$: those that converge to positive values and those that vanish.  

\begin{figure}[h]
\centering
\begin{minipage}{.5\textwidth}
  \centering
  \includegraphics[width=\textwidth]{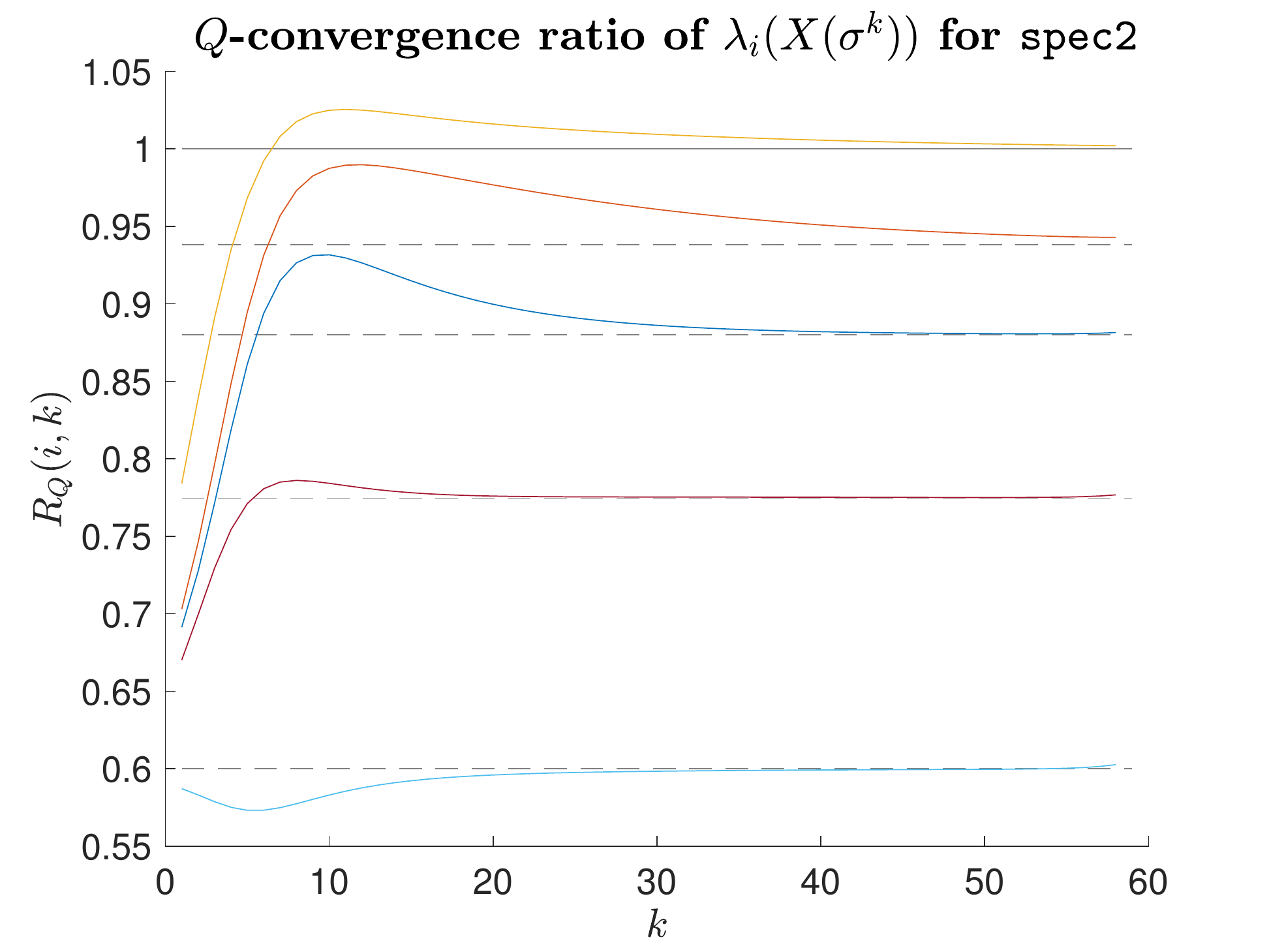}
\end{minipage}%
\begin{minipage}{.5\textwidth}
  \centering
  \includegraphics[width=\textwidth]{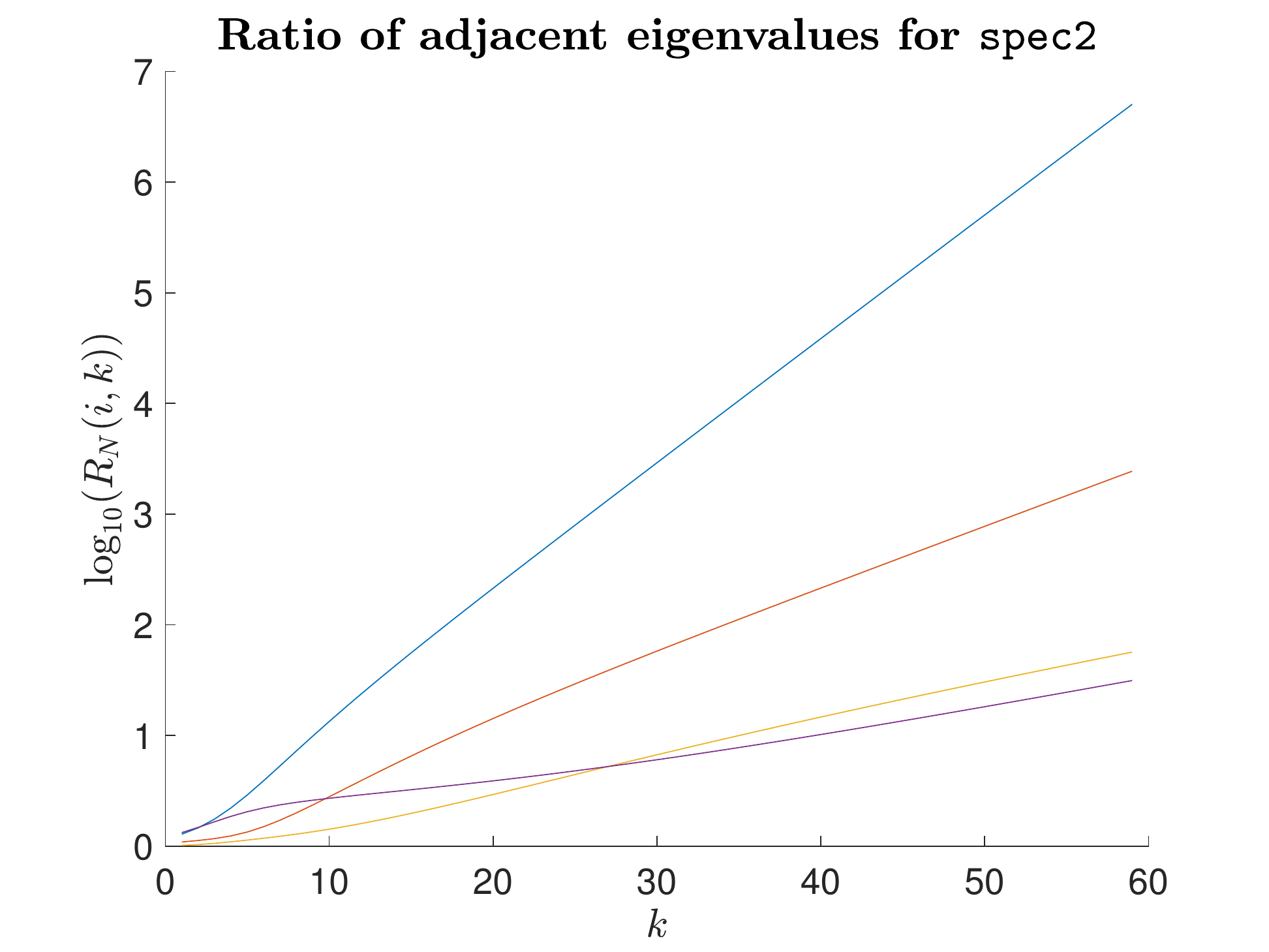}
\end{minipage}
\caption{}
\label{fig:test2}
\end{figure}

The second spectrahedron, denoted \texttt{spec2}, is the classical `worst case' problem as presented in~\cite{MR2724357}, with $n=5$ and singularity degree $n-1=4$.  Plots of the two ratios are in Figure~\ref{fig:test2}.  The left image shows $5$ distinct rates of $Q$-convergence, one for each eigenvalue.  All but one of the curves converge to values that are clearly below $1$.  This indicates, correctly, that the maximum rank is at most $1$.  The largest of the curves appears to converge to the highest of the dashed lines.  Thus we may infer that singularity degree is at least $4$.  Since $4=n-1$, the worst case upper bound, we may conclude that singularity degree is exactly $4$.  The row corresponding to \texttt{spec2} in Table~\ref{table} shows a very large discrepancy between forward error and backward error.  Our lower bound is actually quite close to the true forward error.  Now let us consider the image on the right.  It may be somewhat speculative to assert that the two lower curves blow up.  Thus, taking the more cautious approach we assume that only the two larger curves blow up.  Checking the indices of these curves yields an upper bound of $3$ on the maximum rank.  We choose the notably lower estimate of $1$ based on the left plot.  On the other hand if we are to apply Corollary~\ref{corr:blocktoeig} then we would want an overestimate of the number of different rates of convergence among the eigenvalues of $\Xa$ that vanish.  For this number we include the two lower curves, giving a bound of $4$.

\begin{figure}[h]
\centering
\begin{minipage}{.5\textwidth}
  \centering
  \includegraphics[width=\textwidth]{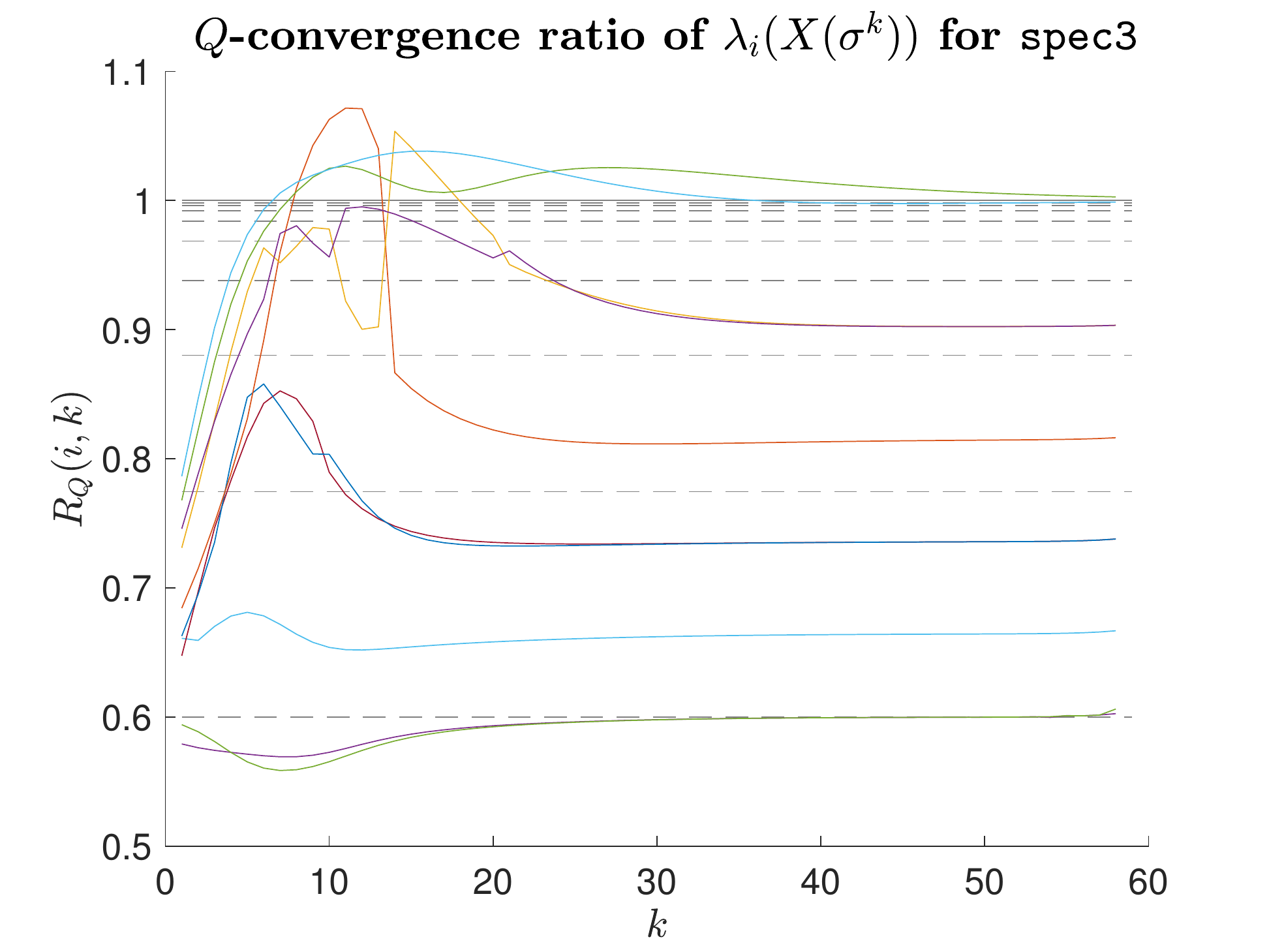}
\end{minipage}%
\begin{minipage}{.5\textwidth}
  \centering
  \includegraphics[width=\textwidth]{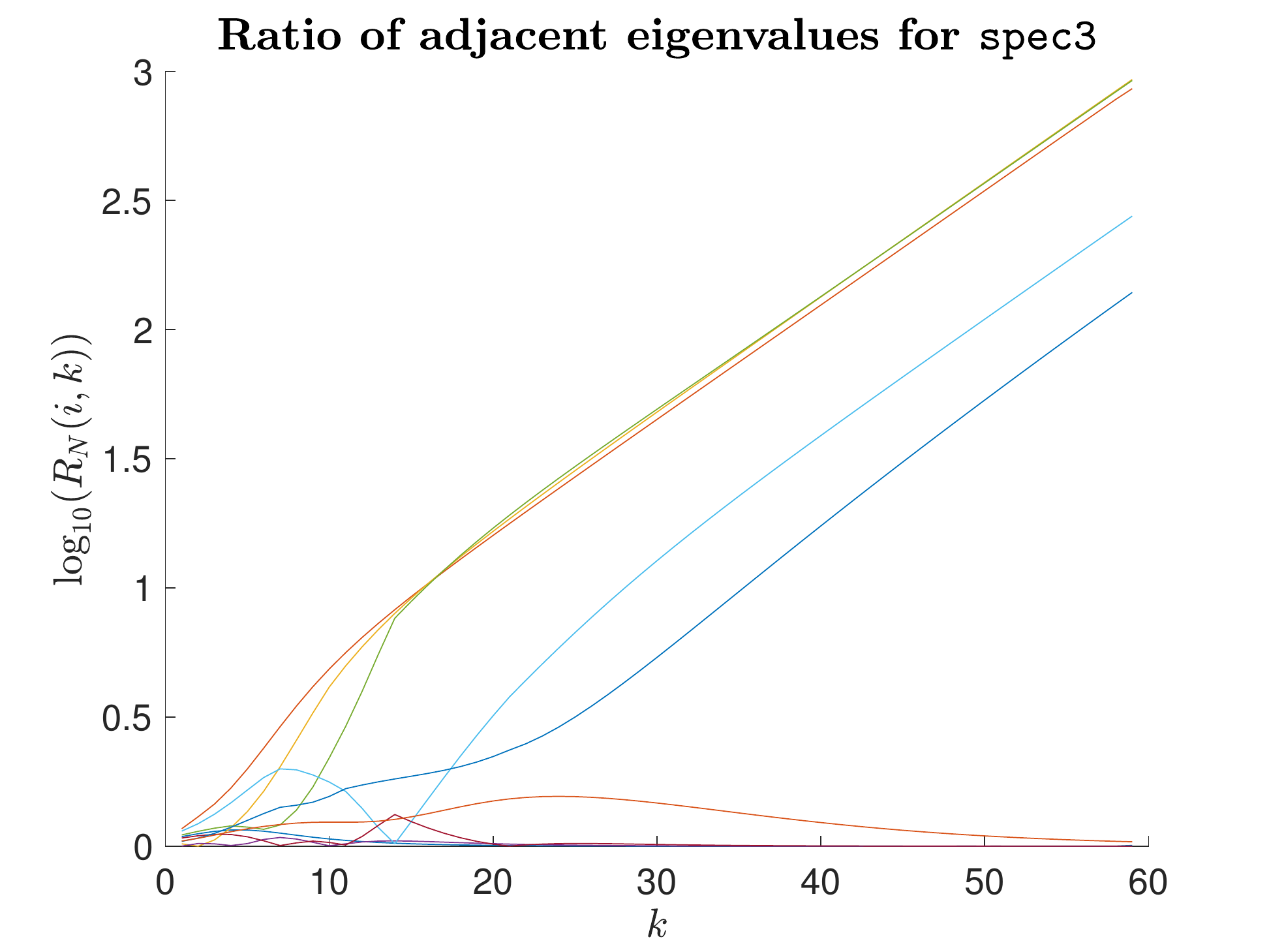}
\end{minipage}
\caption{}
\label{fig:test3}
\end{figure}

For the next spectrahedron, \texttt{spec3}, the authors of~\cite{WaNaMu:12} observed ``strange behaviour" when attempting to optimize over it with an interior point method.  The dimension is $n=10$ and the singularity degree is proven to be $5$.  In Table~\ref{table} we see a large discrepancy between forward error and backward error.  The left image of Figure~\ref{fig:test3} shows six distinct groups of curves.  It is clear, for all but two of the curves, that the limit point is different from $1$.  Thus we have an upper bound of $2$ on the maximum rank.  This upper bound gives quite an accurate lower bound on the forward error. The largest of the curves that does not converge to $1$ appears to converge to a value that is below the fourth dashed line, indicating a lower bound of $4$ on the singularity degree.  Unlike the two previous spectrahedra, here the lower bound on singularity degree is a strict one.  The image on the right shows exactly five different rates of convergence among the eigenvalues of $\Xa$ that converge to $0$.  Moreover the upper bound on maximum rank corresponds to the one obtained from the left image.

\begin{figure}[h]
\centering
\begin{minipage}{.5\textwidth}
  \centering
  \includegraphics[width=\textwidth]{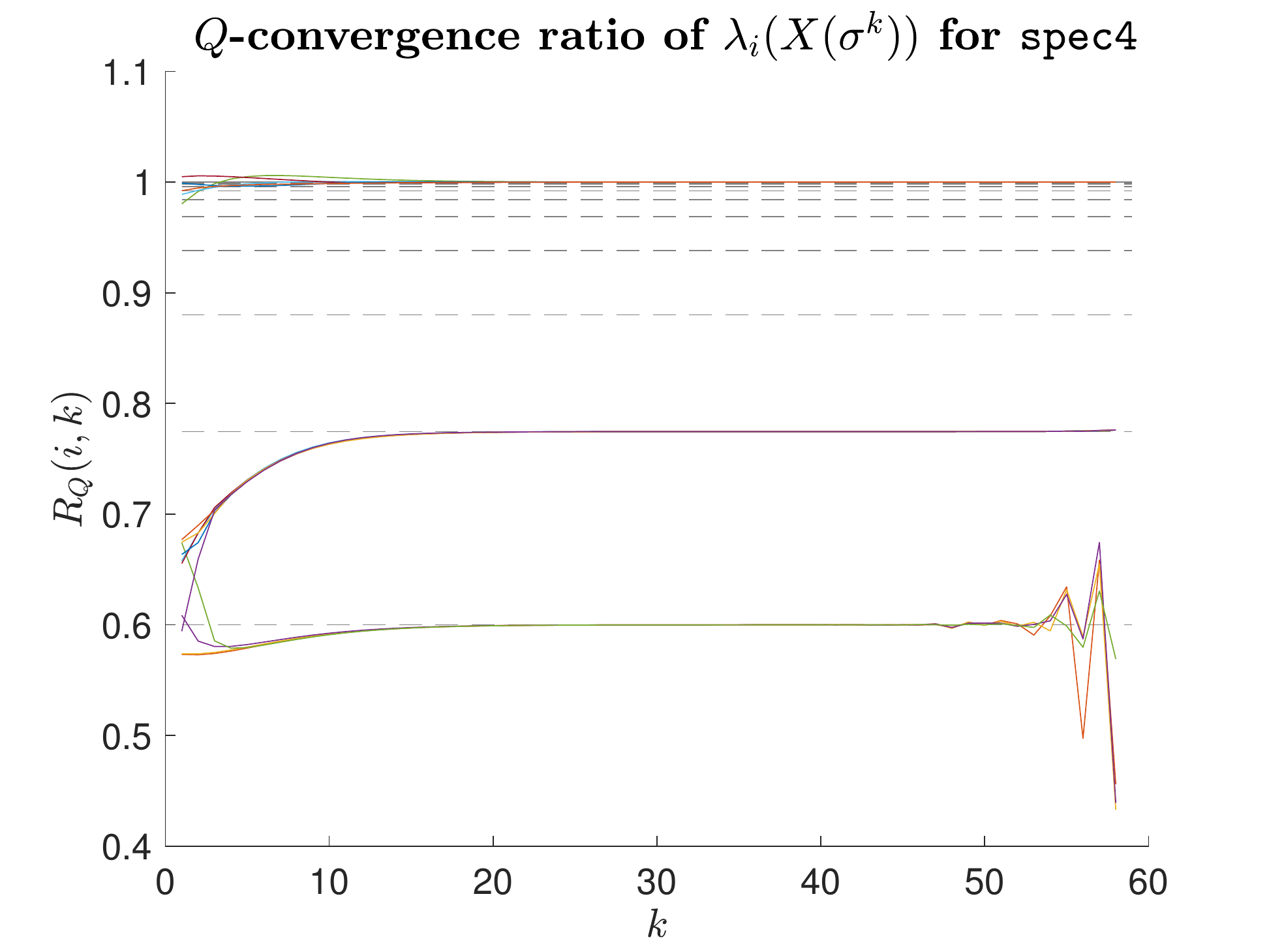}
\end{minipage}%
\begin{minipage}{.5\textwidth}
  \centering
  \includegraphics[width=\textwidth]{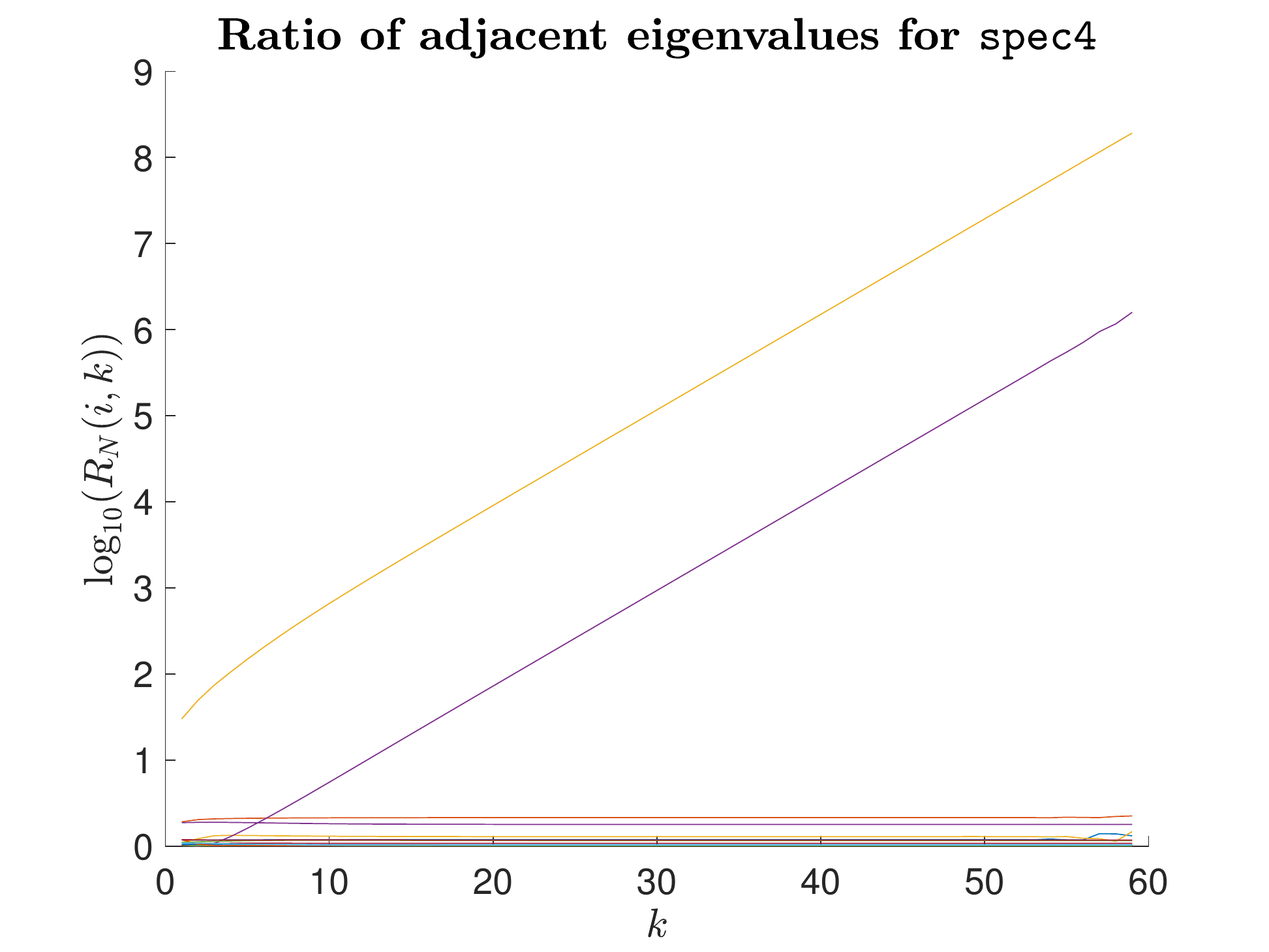}
\end{minipage}
\caption{}
\label{fig:test4}
\end{figure}

The fourth spectrahedron, \texttt{spec4}, is generated by the algorithm of~\cite{WeiWolk:06}, just as \texttt{spec1} is.  However, for this instance we require the existence of a complementarity gap,~i.e.,~strict complementarity does not hold.  While we do not know the exact value of the singularity degree, the lack of strict complementarity gives us a lower bound of $2$.  Plots of the ratios $R_Q(i,k)$ and $R_N(i,k)$ are shown in Figure~\ref{fig:test4}.  From the left image we can be quite sure that those curves that converge to the second dashed line or below do not converge to $1$.  A closer inspection reveals that there are $10$ such curves, implying an upper bound of $5$ on the maximum rank.   The corresponding lower bound on forward error, as recorded in Table~\ref{table}, is indeed a lower bound and more informative than the reported backward error.  We also obtain a lower bound on the singularity degree that coincides with the theoretical lower bound of $2$.  The image on right shows exactly two rates of convergence among eigenvalues of $\Xa$ that converge to $0$ and, once again, provides the same upper bound on maximum rank as obtained from the image on left.

\begin{figure}[h]
\centering
\begin{minipage}{.5\textwidth}
  \centering
  \includegraphics[width=\textwidth]{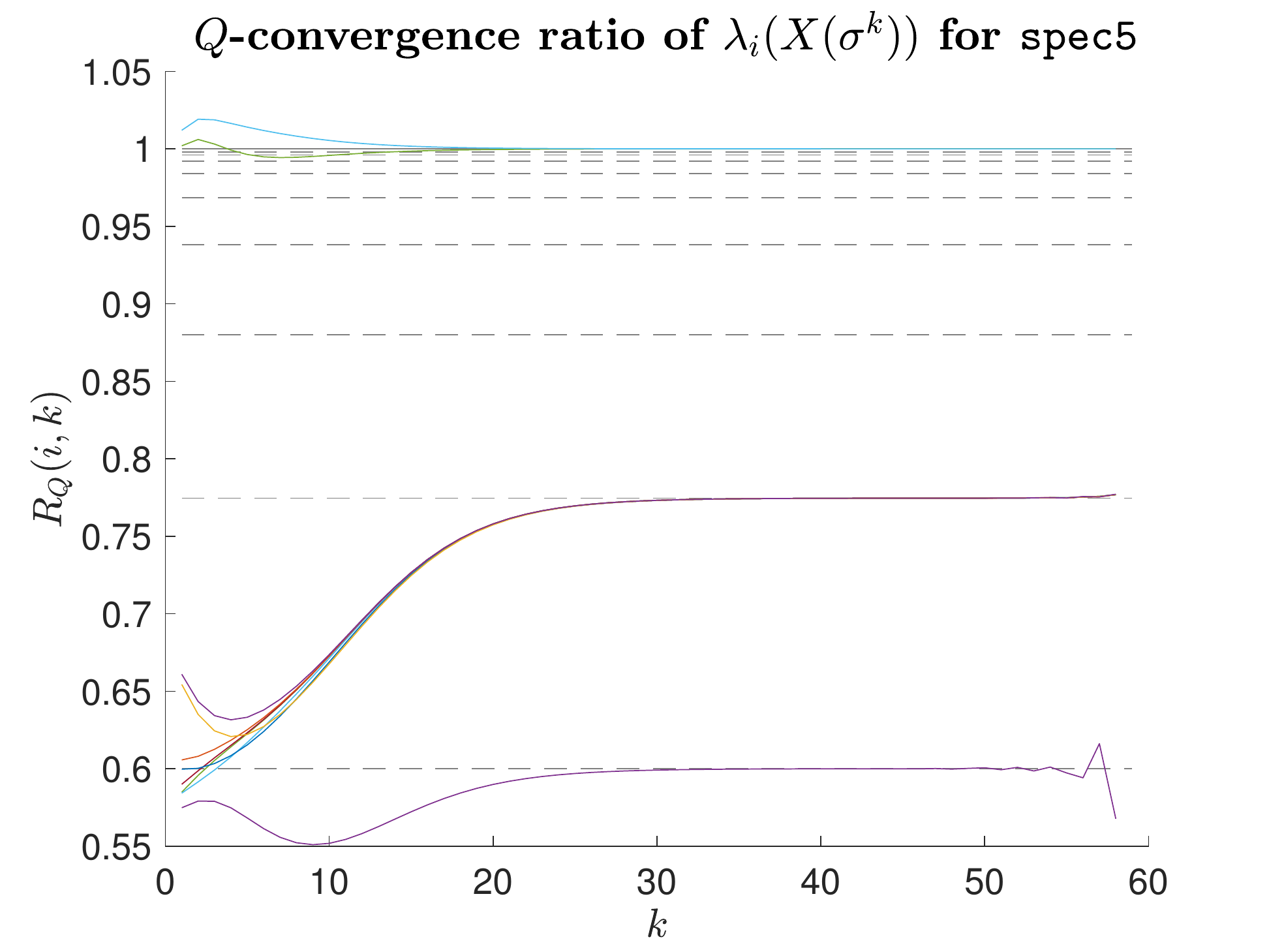}
\end{minipage}%
\begin{minipage}{.5\textwidth}
  \centering
  \includegraphics[width=\textwidth]{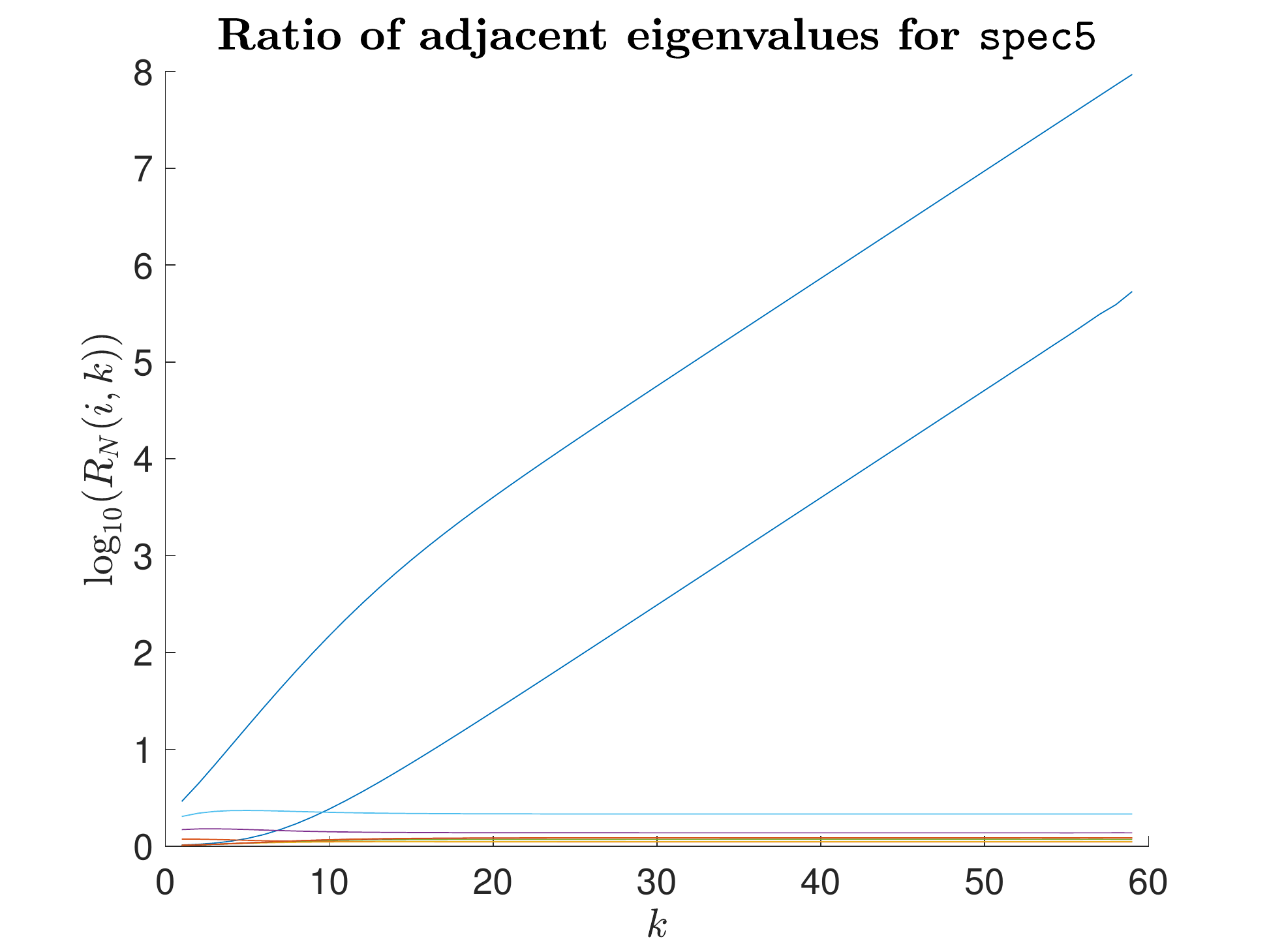}
\end{minipage}
\caption{}
\label{fig:test5}
\end{figure}

The final spectrahedron we consider is a \emph{Toeplitz cycle completion problem} of the form introduced in Corollary~6 of~\cite{MR1236734}.  The specific instance we consider here is that of Example~4.5 of~\cite{SWWb:17} with $n=10$.  While the singularity degree for this problem is not known, a lower bound of $2$ was proven in~\cite{SWWb:17}.  In Figure~\ref{fig:test5}, we find images of plots of the ratios $R_Q(i,k)$ and $R_N(i,k)$.  The left image indicates that all but two of the eigenvalues of $\Xa$ converge to $0$, yielding  an exact approximation of maximum rank.  Moreover, the corresponding eight curves appear to have limits below the second dashed line.  Hence we have a lower bound of $2$ on the singularity degree.  This coincides with the theoretical lower bound.

\begin{table}[h]
\centering
\caption{A record of relevant measures and their bounds for the spectrahedra considered in our analysis. Here $\overline r$, \underline{$\epsilon$}, and \underline{$d$} denote the upper bound on rank, lower bound on forward error, and lower bound on singularity degree, respectively.  The number of different rates of convergence among the eigenvalues of $\Xa$ that vanish is denoted by $N_{\lambda}$.}
\label{table}
\begin{tabular}{| l | c | c | c | c | c | c | c | c | c |}
\hline 
$\F$ & $\epsilon^b(\F)$ & $r$ & $\overline r$ &  $\epsilon^f(\F)$ & \underline{$\epsilon$} & $\sd(\F)$ & \underline{$d$} & $N_{\lambda}$ \\
\hline 
\texttt{spec1} & $6.62 \times 10^{-11}$ & $7$ & $7$ & $4.36\times 10^{-11}$ & $3.10 \times 10^{-12}$ & $1$ & $1$ & $1$ \\
\texttt{spec2} & $4.44 \times 10^{-13}$ & $1$ & $1$ & $4.93 \times 10^{-2}$ & $3.19\times 10^{-2}$ & $4$ & $4$ & $4$ \\ 
\texttt{spec3} & $2.47\times 10^{-13}$ & $2$ & $2$ & $3.39\times 10^{-2}$ & $9.88\times 10^{-3}$ & $5$ & $4$ & $5$ \\
\texttt{spec4} & $1.88 \times 10^{-11}$ & $5$ & $5$ & $1.35 \times 10^{-5}$ & $4.16 \times 10^{-7}$ & $\ge 2$ & $2$ & $2$ \\
\texttt{spec5} & $2.61\times 10^{-13}$ & $2$ & $2$ & $1.22\times 10^{-6}$ & $6.96 \times 10^{-8}$ & $\ge 2$ & $2$ & $2$ \\
\hline
\end{tabular}
\end{table}

In these case studies we have demonstrated the ability to upper bound maximum rank quite effectively.  The resulting lower bound on forward error is of a much larger magnitude than backward error in all instances with the exception of \texttt{spec1}, where the singularity degree is~$1$.  We see this feature as quite useful, as it alerts practitioners that the proposed solution is of substantially lower accuracy than backward error indicates.  For spectrahedra with known singularity degree, we have demonstrated that the lower bound is quite accurate.  In the other cases, the lower bound is in agreement with the theoretical lower bound.  Lastly, for these test cases (as well as for others we have tested), the value $N_{\lambda}$ seems to be an upper bound on singularity degree. Proving this, or demonstrating a counterexample, is an interesting topic for future research.

\cleardoublepage

\bibliographystyle{plain}
\bibliography{../../mytexmf/.master,../../mytexmf/.psd,../../mytexmf/.edm,../../mytexmf/.bjorBOOK,../../mytexmf/.qap} 

\def\cprime{$'$} \def\cprime{$'$} \def\cprime{$'$}
  \def\udot#1{\ifmmode\oalign{$#1$\crcr\hidewidth.\hidewidth
  }\else\oalign{#1\crcr\hidewidth.\hidewidth}\fi} \def\cprime{$'$}
  \def\cprime{$'$} \def\cprime{$'$}
\begin{thebibliography}{10}

\bibitem{MR1236734}
W.~Barrett, C.R. Johnson, and P.~Tarazaga.
\newblock The real positive definite completion problem for a simple cycle.
\newblock {\em Linear Algebra Appl.}, 192:3--31, 1993.
\newblock Computational linear algebra in algebraic and related problems
  (Essen, 1992).

\bibitem{bw2}
J.M. Borwein and H.~Wolkowicz.
\newblock Characterization of optimality for the abstract convex program with
  finite-dimensional range.
\newblock {\em J. Austral. Math. Soc. Ser. A}, 30(4):390--411, 1980/81.

\bibitem{bw1}
J.M. Borwein and H.~Wolkowicz.
\newblock Facial reduction for a cone-convex programming problem.
\newblock {\em J. Austral. Math. Soc. Ser. A}, 30(3):369--380, 1980/81.

\bibitem{bw3}
J.M. Borwein and H.~Wolkowicz.
\newblock Regularizing the abstract convex program.
\newblock {\em J. Math. Anal. Appl.}, 83(2):495--530, 1981.

\bibitem{ScTuWonumeric:07}
Y-L. Cheung, S.~Schurr, and H.~Wolkowicz.
\newblock Preprocessing and regularization for degenerate semidefinite
  programs.
\newblock In D.H. Bailey, H.H. Bauschke, P.~Borwein, F.~Garvan, M.~Thera,
  J.~Vanderwerff, and H.~Wolkowicz, editors, {\em Computational and
  {A}nalytical {M}athematics, {I}n {H}onor of {J}onathan {B}orwein's 60th
  {B}irthday}, volume~50 of {\em Springer Proceedings in Mathematics \&
  Statistics}, pages 225--276. Springer, 2013.

\bibitem{KlerkRoosTerlakyinit:97}
E.~de~Klerk, C.~Roos, and T.~Terlaky.
\newblock Initialization in semidefinite programming via a self-dual
  skew-symmetric embedding.
\newblock {\em Oper. Res. Lett.}, 20(5):213--221, 1997.

\bibitem{int:deklerk7}
E.~de~Klerk, C.~Roos, and T.~Terlaky.
\newblock Infeasible--start semidefinite programming algorithms via self--dual
  embeddings.
\newblock In {\em Topics in Semidefinite and Interior-Point Methods}, volume~18
  of {\em The Fields Institute for Research in Mathematical Sciences,
  Communications Series}, pages 215--236. American Mathematical Society, 1998.

\bibitem{KrukDoanW:10}
X.V. Doan, S.~Kruk, and H.~Wolkowicz.
\newblock A robust algorithm for semidefinite programming.
\newblock {\em Optim. Methods Softw.}, 27(4-5):667--693, 2012.

\bibitem{DrusWolk:16}
D.~Drusvyatskiy and H.~Wolkowicz.
\newblock The many faces of degeneracy in conic optimization.
\newblock {\em Foundations and Trends\textsuperscript{\tiny\textregistered} in
  Optimization}, 3(2):77--170, 2017.

\bibitem{Fan:50}
K.~Fan.
\newblock On a theorem of {W}eyl concerning eigenvalues of linear
  transformations ii.
\newblock {\em Proc.\ Nat.\ Acad.\ Sci.\ U.S.A.}, 36:31--35, 1950.

\bibitem{int:Goldfarb15}
D.~Goldfarb and K.~Scheinberg.
\newblock Interior point trajectories in semidefinite programming.
\newblock {\em SIAM J. Optim.}, 8(4):871--886, 1998.

\bibitem{Halicka:01}
M.~Halick{\'a}.
\newblock Analyticity of the central path at the boundary point in semidefinite
  programming.
\newblock {\em European J. Oper. Res.}, 143(2):311--324, 2002.
\newblock Interior point methods (Budapest, 2000).

\bibitem{HalickaKlerkRoos:01}
M.~Halick{\'a}, E.~de~Klerk, and C.~Roos.
\newblock On the convergence of the central path in semidefinite optimization.
\newblock {\em SIAM J. Optim.}, 12(4):1090--1099 (electronic), 2002.

\bibitem{hw53}
A.J. Hoffman and H.W. Wielandt.
\newblock The variation of the spectrum of a normal matrix.
\newblock {\em Duke Mathematics}, 20:37--39, 1953.

\bibitem{KrMuReVaWo:98}
S.~Kruk, M.~Muramatsu, F.~Rendl, R.J. Vanderbei, and H.~Wolkowicz.
\newblock The {G}auss-{N}ewton direction in semidefinite programming.
\newblock {\em Optim. Methods Softw.}, 15(1):1--28, 2001.

\bibitem{lusz00}
Z-Q. Luo, J.F. Sturm, and S.~Zhang.
\newblock Conic convex programming and self-dual embedding.
\newblock {\em Optim. Methods Softw.}, 14(3):169--218, 2000.

\bibitem{mi68}
J.~Milnor.
\newblock {\em Singular points of complex hypersurfaces}.
\newblock Annals of Mathematics Studies, No. 61. Princeton University Press,
  Princeton, N.J.; University of Tokyo Press, Tokyo, 1968.

\bibitem{NestToddYe:96}
Y.E. Nesterov, M.J. Todd, and Y.~Ye.
\newblock Infeasible-start primal-dual methods and infeasibility detectors for
  nonlinear programming problems.
\newblock {\em Math. Program.}, 84(2, Ser. A):227--267, 1999.

\bibitem{MR3108446}
G.~Pataki.
\newblock Strong duality in conic linear programming: facial reduction and
  extended duals.
\newblock In David Bailey, Heinz~H. Bauschke, Frank Garvan, Michel Thera,
  Jon~D. Vanderwerff, and Henry Wolkowicz, editors, {\em Computational and
  analytical mathematics}, volume~50 of {\em Springer Proc. Math. Stat.}, pages
  613--634. Springer, New York, 2013.

\bibitem{permfribergandersen}
F.~Permenter, H.~Friberg, and E.~Andersen.
\newblock Solving conic optimization problems via self-dual embedding and
  facial reduction: a unified approach.
\newblock Technical report, MIT, Boston, MA, 2015.

\bibitem{perm}
F.~Permenter and P.~Parrilo.
\newblock Partial facial reduction: simplified, equivalent {SDP}s via
  approximations of the {PSD} cone.
\newblock Technical Report Preprint arXiv:1408.4685, MIT, Boston, MA, 2014.

\bibitem{PotShe:95}
F.A. Potra and R.~Sheng.
\newblock A superlinearly convergent primal-dual infeasible-interior-point
  algorithm for semidefinite programming.
\newblock Technical Report Reports on Computational Mathematics, 78, University
  of Iowa, Iowa City, IA, 1995.

\bibitem{con:70}
R.T. Rockafellar.
\newblock {\em Convex analysis}.
\newblock Princeton Mathematical Series, No. 28. Princeton University Press,
  Princeton, N.J., 1970.

\bibitem{SWWb:17}
S.~Sremac, H.J. Woerdeman, and H.~Wolkowicz.
\newblock Maximum determinant positive definite {T}oeplitz completions.
\newblock In {\em Operator Theory, Analysis and the State Space Approach: In
  Honor of Rien Kaashoek}, volume 271, pages 421--441. Birkh{\"a}user/Springer,
  Cham, 2018.

\bibitem{S98lmi}
J.F. Sturm.
\newblock Error bounds for linear matrix inequalities.
\newblock {\em SIAM J. Optim.}, 10(4):1228--1248 (electronic), 2000.

\bibitem{MR2724357}
L.~Tun{\c{c}}el.
\newblock {\em Polyhedral and Semidefinite Programming Methods in Combinatorial
  Optimization}, volume~27 of {\em Fields Institute Monographs}.
\newblock American Mathematical Society, Providence, RI, 2010.

\bibitem{MR3063940}
H.~Waki and M.~Muramatsu.
\newblock Facial reduction algorithms for conic optimization problems.
\newblock {\em J. Optim. Theory Appl.}, 158(1):188--215, 2013.

\bibitem{WaNaMu:12}
H.~Waki, M.~Nakata, and M.~Muramatsu.
\newblock Strange behaviors of interior-point methods for solving semidefinite
  programming problems in polynomial optimization.
\newblock {\em Computational Optimization and Applications}, 53(3):823--844,
  2012.

\bibitem{WeiWolk:06}
H.~Wei and H.~Wolkowicz.
\newblock Generating and measuring instances of hard semidefinite programs.
\newblock {\em Math. Program.}, 125(1, Ser. A):31--45, 2010.

\bibitem{SaVaWo:97}
H.~Wolkowicz, R.~Saigal, and L.~Vandenberghe, editors.
\newblock {\em Handbook of semidefinite programming}.
\newblock International Series in Operations Research \& Management Science,
  27. Kluwer Academic Publishers, Boston, MA, 2000.
\newblock Theory, algorithms, and applications.

\end{thebibliography}
\addcontentsline{toc}{section}{Bibliography}

\end{document}